\documentclass[10pt,twocolumn,twoside] {IEEEtran}

\usepackage{times}
\usepackage{epsfig}
\usepackage{graphicx}
\usepackage{amsmath}
\usepackage{amssymb}
\usepackage{tikz}
\usepackage{algpseudocode}
\usepackage{algorithm}
\usepackage{mathrsfs}

\newtheorem{theorem}{Theorem}
\newtheorem{lemma}{Lemma}
\newtheorem{proposition}{Proposition}
\newtheorem{corollary}{Corollary}
\newtheorem{definition}{Definition}
\newtheorem{remark}{Remark}

\def\QED{~\rule[-1pt]{5pt}{5pt}\par\medskip}
\newenvironment{proof}{\emph{Proof.}}{\hfill\QED}

\renewcommand{\S}{\mathcal{S}}
\newcommand{\A}{\boldsymbol{A}}
\newcommand{\B}{\boldsymbol{B}}
\newcommand{\D}{\boldsymbol{D}}
\newcommand{\E}{\boldsymbol{E}}
\newcommand{\G}{\boldsymbol{G}}

\newcommand{\I}{\boldsymbol{I}}
\renewcommand{\a}{\boldsymbol{a}}
\renewcommand{\b}{\boldsymbol{b}}
\renewcommand{\c}{\boldsymbol{c}}

\newcommand{\e}{\boldsymbol{e}}

\newcommand{\s}{\boldsymbol{s}}
\newcommand{\x}{\boldsymbol{x}}
\newcommand{\y}{\boldsymbol{y}}
\newcommand{\z}{\boldsymbol{z}}

\renewcommand{\Re}{\mathbb{R}}
\newcommand{\0}{\boldsymbol{0}}

\newcommand{\ie}{\text{i.e.}}
\newcommand{\eg}{\text{e.g.}}

\newcommand{\rank}{\operatorname{rank}}

\newcommand{\st}{\operatorname{s. t.}}
\newcommand{\spn}{\operatorname{span}}
\newcommand{\argmin}{\operatorname{argmin}}
\newcommand{\myparagraph}[1]{\medskip\noindent\textbf{#1.} }

\def\x{{\mathbf x}}

\begin{document}

\title{Block-Sparse Recovery via Convex Optimization}

\author{Ehsan~Elhamifar,~\IEEEmembership{Student~Member,~IEEE,}
        and~Ren\'e~Vidal,~\IEEEmembership{Senior Member,~IEEE}
\thanks{E. Elhamifar is with the Department
of Electrical and Computer Engineering, The Johns Hopkins University, 3400 N.
Charles St., Baltimore MD 21218 USA. E-mail: ehsan@cis.jhu.edu. 
}
\thanks{R. Vidal is with the Center for Imaging Science, Department of Biomedical
Engineering, The Johns Hopkins University, 302B Clark Hall, 3400 N.
Charles St., Baltimore MD 21218 USA. E-mail: rvidal@cis.jhu.edu.}
}

\maketitle

\vspace{-1mm}

\begin{abstract}
Given a dictionary that consists of multiple blocks and a signal that lives in the range space of only a few blocks, we study the problem of finding a block-sparse representation of the signal, i.e., a representation that uses the minimum number of blocks. Motivated by signal/image processing and computer vision applications, such as face recognition, we consider the block-sparse recovery problem in the case where the number of atoms in each block is arbitrary, possibly much larger than the dimension of the underlying subspace. To find a block-sparse representation of a signal, we propose two classes of non-convex optimization programs, which aim to minimize the number of nonzero coefficient blocks and the number of nonzero reconstructed vectors from the blocks, respectively. Since both classes of problems are NP-hard, we propose convex relaxations and derive conditions under which each class of the convex programs is equivalent to the original non-convex formulation. Our conditions depend on the notions of mutual and cumulative \emph{subspace} coherence of a dictionary, which are natural generalizations of existing notions of mutual and cumulative coherence. We evaluate the performance of the proposed convex programs through simulations as well as real experiments on face recognition. We show that treating the face recognition problem as a block-sparse recovery problem improves the state-of-the-art results by $10\%$ with only $25\%$ of the training data.

\end{abstract}


\IEEEpeerreviewmaketitle


\section{Introduction}

\subsection{Recovery of Sparse Signals}

Sparse signal recovery has drawn increasing attention in many areas such as signal/image processing, computer vision, machine learning, and bioinformatics (see \eg, \cite{Candes:SPM08, Wright:IEEEProc10, Elad:IEEEProc10, Parvaresh:STSP08} and the references therein). The key assumption behind sparse signal recovery is that an observed signal $\y$ can be written as a linear combination of a few atoms of a given dictionary $\B$. 


More formally, consider an underdetermined system of linear equations of the form $\y = \B \c$, where $\B \in \Re^{D \times N}$ has more columns than rows, hence allowing infinitely many solutions for a given $\y$.  Sparsity of the desired solution  
arises in many problems and can be used to restrict the set of possible solutions. In principle, the problem of finding the sparsest representation of a given signal can be cast as the following optimization program 
\begin{equation}
\label{eq:PL0}
P_{\ell_0}: \; \min \| \c \|_0 \quad \st \quad \y = \B \c ,
\end{equation}
where $\| \c \|_0$ is the $\ell_0$ quasi-norm of $\c$, which counts the number of nonzero elements of $\c$. We say that a vector $\c$ is $k$-sparse if it has at most $k$ nonzero elements. While finding the sparse representation of a given signal using $P_{\ell_0}$ is NP-hard \cite{Amaldi:TCS98}, the pioneering work of Donoho \cite{Donoho:CPAM06} and Candes \cite{Candes-Tao:TIT05} showed that, under appropriate conditions, this problem can be solved efficiently as\!
\begin{equation}
\label{eq:PL1}
P_{\ell_1}: \; \min \| \c \|_1 \quad \st \quad \y = \B \c.
\end{equation}
%
Since then, there has been an outburst of research articles addressing conditions under which the two optimization programs, $P_{\ell_1}$ and $P_{\ell_0}$, are equivalent. Most of these results are based on the notions of mutual/cumulative coherence \cite{Tropp:TIT04, Donoho:TIT06} and restricted isometry property \cite{Candes-Tao:TIT05, Candes:CRAS08}, which we describe next. Throughout the paper, we assume that the columns of $\B$ have unit Euclidean norm.

\myparagraph{Mutual/Cumulative Coherence} 
The mutual coherence of a dictionary $\B$ is defined as 
\begin{equation}
\label{eq:mutcoh}
\mu \triangleq \max_{i \neq j} | \b_i^{\top} \b_j |,
\end{equation}
where $\b_i$ denotes the $i$-th column of $\B$ of unit Euclidean norm. \cite{Tropp:TIT04} and \cite{Donoho:TIT06} show that if the sufficient condition 
\begin{equation}
\label{eq:mut-coh-conventional}
(2k-1) \mu < 1,
\end{equation} 
holds, then the optimization programs $P_{\ell_1}$ and $P_{\ell_0}$ are equivalent and recover the $k$-sparse representation of a given signal. While $\mu$ can be easily computed, it does not characterize a dictionary very well since it measures the most extreme correlations in the dictionary. 

To better characterize a dictionary, cumulative coherence measures the maximum total coherence between a fixed atom and a collection of $k$ other atoms. Specifically, the cumulative coherence associated to a positive integer $k$ \cite{Tropp:TIT04} is defined as 
\begin{equation}
\label{eq:cumcoh}
\zeta_k \triangleq \max_{\Lambda_k} \max_{i \notin \Lambda_k} \sum_{j \in \Lambda_k} | \b_i^{\top} \b_j |,
\end{equation}
where $\Lambda_k$ denotes a set of $k$ different indices from $\{1, \ldots, N \}$. Note that for $k=1$, we have $\zeta_1 = \mu$. 
Although cumulative coherence is, in general, more difficult to compute than mutual coherence, it provides sharper results for the equivalence of $P_{\ell_1}$ and $P_{\ell_0}$. In particular, \cite{Tropp:TIT04} shows that if 
\begin{equation}
\label{eq:cum-coh-conventional}
\zeta_k + \zeta_{k-1} < 1,
\end{equation}
then the optimization programs $P_{\ell_1}$ and $P_{\ell_0}$ are equivalent and recover the $k$-sparse representation of a given signal. 

\myparagraph{Restricted Isometry Property} 
An alternative sufficient condition for the equivalence between $P_{\ell_1}$ and $P_{\ell_0}$ is based on the so-called restricted isometry property (RIP) \cite{Candes-Tao:TIT05,Candes:CRAS08}. For a positive integer $k$, the restricted isometry constant of a dictionary $\B$ is defined as the smallest constant $\delta_k$ for which
\begin{equation}
\label{eq:RIP}
(1 - \delta_k) \| \c \|_2^2 \leq \| \B \c \|_2^2 \leq (1 + \delta_k) \| \c \|_2^2
\end{equation}
holds for all $k$-sparse vectors $\c$. \cite{Candes:CRAS08} shows that if $\delta_{2k} < \sqrt{2}-1$, then $P_{\ell_1}$ and $P_{\ell_0}$ are equivalent. The bound in this result has been further improved and \cite{Foucart:ACHA10} shows that if $\delta_{2k} < 0.4652$, then $P_{\ell_1}$ and $P_{\ell_0}$ are equivalent.

\subsection{Recovery of Block-Sparse Signals}

Recently, there has been growing interest in recovering sparse representations of signals in a union of a large number of subspaces, under the assumption that the signals live in the direct sum of only a few subspaces. Such a representation whose nonzero elements appear in a few blocks is called a \emph{block-sparse} representation. Block sparsity arises in various applications such as reconstructing multi-band signals \cite{Mishali:TSP09, Eldar:STSP10}, measuring gene expression levels \cite{Parvaresh:STSP08}, face/digit/speech recognition \cite{Wright:PAMI09, Elhamifar:CVPR11, Gemmeke:ISCA08,Gemmeke:EUSIPCO08}, clustering of data on multiple subspaces \cite{Elhamifar:CVPR09, Elhamifar:ICASSP10, Elhamifar:TPAMI12, Candes:TechRep11}, finding exemplars in datasets \cite{Elhamifar:CVPR12}, multiple measurement vectors \cite{Cotter:TSP05, Chen:TSP06, VandenBerg:TIT10, Lai:ACHA11}, etc.


The recovery of block-sparse signals involves solving a system of linear equations of the form 
\begin{equation}
\label{eq:yBc}
\y = \B \c = \begin{bmatrix} \B[1] & \cdots & \B[n] \end{bmatrix} \c,
\end{equation}
where $\B$ consists of $n$ blocks $\B[i] \in \Re^{D \times m_i}$. 
%
%
The main difference with respect to classical sparse recovery is that the desired solution of \eqref{eq:yBc} corresponds to a few nonzero \emph{blocks} rather than a few nonzero \emph{elements} of $\B$. We say that a vector $\c^{\top} = \begin{bmatrix} \c[1]^{\top} & \cdots & \c[n]^{\top} \end{bmatrix}$ is $k$-\emph{block-sparse}, if at most $k$ blocks $\c[i] \in \Re^{m_i}$ are different from zero. Note that, in general, a block-sparse vector is not necessarily sparse and vice versa, as shown in Figure \ref{fig:blk}.

\begin{figure}[t]
\centering
\includegraphics[width=0.9\linewidth, trim = 0 0 0 0 , clip]{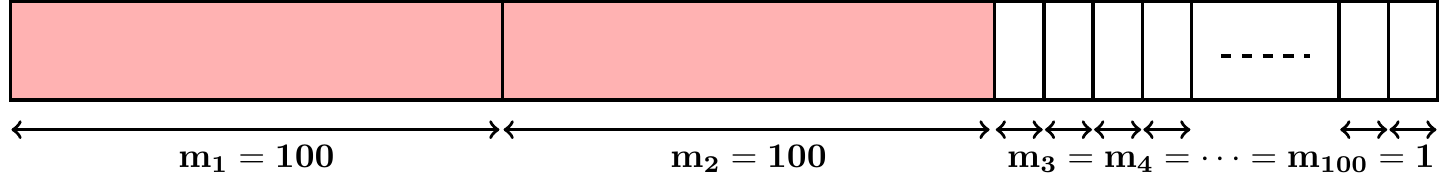} \\
\vspace{2mm}
\hspace{.1mm}
\includegraphics[width=0.9\linewidth, trim = 0 0 0 0 , clip]{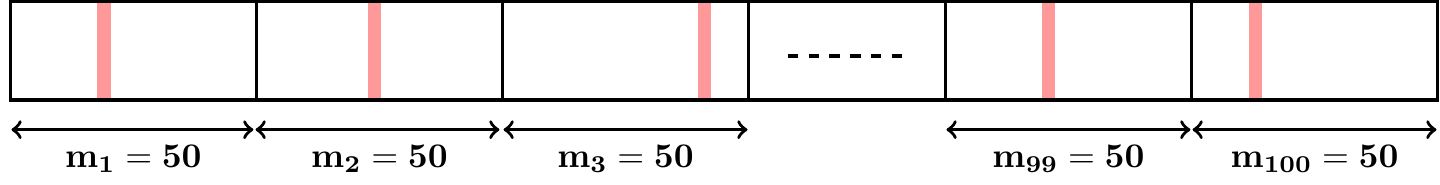}
\vspace{-1mm}
\caption{\footnotesize{Top: a block-sparse vector is not necessarily sparse. In this example, $2$ nonzero blocks out of $100$ blocks correspond to $200$ nonzero elements out of $298$ elements. Bottom: a sparse vector is not necessarily block-sparse. In this example, all $100$ blocks are nonzero each having one nonzero element. However, this gives rise to only 50 nonzero elements out of 5,000 elements.}}
\label{fig:blk}
\end{figure}
%


The problem of finding a representation of a signal $\y$ that uses the minimum number of blocks of $\B$ can be cast as the following optimization program
\begin{equation}
\label{eq:L2L0-A}
P_{\ell_q/\ell_0}: \; \min \sum_{i=1}^{n}{I(\| \c[i] \|_q)} \quad \st \quad \y = \B \c,
\end{equation}
where $q \geq 0$ and $I(\cdot)$ is the indicator function, which is zero when its argument is zero and is one otherwise. In fact, the objective function in \eqref{eq:L2L0-A} counts the number of nonzero blocks of a solution. However, solving  \eqref{eq:L2L0-A} is an NP-hard problem as it requires searching exhaustively over all choices of a few blocks of $\B$ and checking whether they span the observed signal. The $\ell_1$ relaxation of $P_{\ell_q/\ell_0}$ has the following form
%
%
\begin{equation}
\label{eq:L2L1-A}
P_{\ell_q/\ell_1}: \; \min \sum_{i=1}^{n}{\| \c[i] \|_q} \quad \st \quad \y = \B \c.
\end{equation}
For $q \geq 1$, the optimization program $P_{\ell_q/\ell_1}$ is convex and can be solved efficiently using convex programming tools \cite{BoydVandenberghe04}. 
\vspace{1mm}
\begin{remark}
For $q=1$, the convex program $P_{\ell_1/\ell_1}$ is the same as $P_{\ell_1}$ in \eqref{eq:PL1} used for sparse recovery. In other words, while the $\ell_1$ optimization program, under some conditions, can recover a sparse representation of a signal, it can also recover a block-sparse representation, under appropriate conditions, as we will discuss in this paper. 
\end{remark}
\vspace{1mm}

The works of \cite{Stojnic:TSP09, Eldar:TSP10, Eldar:TIT09} study conditions under which for the special case of $q=2$, $P_{\ell_2/\ell_1}$ and $P_{\ell_2/\ell_0}$ are equivalent.
These conditions are based on generalizations of mutual coherence and restricted isometry constant, as described next.

\myparagraph{Block-Coherence} The work of \cite{Eldar:TSP10} assumes that the blocks have linearly independent columns and are of the same length $d$,
\ie, $\rank(\B[i]) = m_i = d$. Under these assumptions, \cite{Eldar:TSP10} defines the block-coherence of a dictionary $\B$ as 
\begin{equation}
\label{eq:block-coh}
\mu_{B} \triangleq \max_{i \neq j} \frac{1}{d} \sigma_1(\B[i]^{\top}\B[j]),
\end{equation}
where $\sigma_1(\cdot)$ denotes the largest singular value of the given matrix. Also, the subcoherence of $\B$ is defined as $\nu \triangleq \max_{i}{\mu_i}$ where $\mu_i$ denotes the mutual coherence for the $i$-th block. \cite{Eldar:TSP10} shows that if 
\begin{equation}
\label{eq:block-coh-cond}
(2k-1) d \mu_{B} < 1 - (d-1) \nu,
\end{equation}
then $P_{\ell_2/\ell_1}$ and $P_{\ell_2/\ell_0}$ are equivalent and recover the $k$-block-sparse representation of a given signal. 

\myparagraph{Block-RIP}  \cite{Eldar:TIT09} assumes that the blocks have linearly independent columns, although their lengths need not be equal. Under this assumption, \cite{Eldar:TIT09} defines the block restricted isometry constant of $\B$ as the smallest constant $\delta_{B,k}$ such that
\begin{equation}
\label{eq:block-RIP}
(1 - \delta_{B,k}) \| \c \|_2^2 \leq \| \B \c \|_2^2 \leq (1 + \delta_{B,k}) \| \c \|_2^2
\end{equation}
holds for every $k$-block-sparse vector $\c$. Analogous to the conventional sparse recovery results, \cite{Eldar:TIT09} shows that if $\delta_{B,2k} < \sqrt{2} - 1$, then $P_{\ell_2/\ell_1}$ and $P_{\ell_2/\ell_0}$ are equivalent. 

The work of \cite{Stojnic:TSP09} proposes an alternative analysis framework for block-sparse recovery using $P_{\ell_2/\ell_1}$ in the special case of Gaussian dictionaries. By analyzing the nullspace of the dictionary, it shows that if the blocks have linearly independent columns, perfect recovery is achieved with high probability as the length of the signal, $D$, grows to infinity.

An alternative approach to recover the block-sparse representation of a given signal is to solve the optimization program 
\begin{equation}
\label{eq:L2L0-B}
P'_{\ell_q/\ell_0}\!\!: \, \min \sum_{i=1}^{n}{I(\| \B[i] \c[i] \|_q)} \quad \st \quad \y = \B \c,\!
\end{equation}
for $q \geq 0$. Notice that the solution to this problem coincides with that of $P_{\ell_q/\ell_0}$ for blocks with linearly independent columns since $\| \B[i] \c[i] \|_q > 0$ if and only if $\| \c[i] \|_q > 0$. Nevertheless, $P'_{\ell_q/\ell_0}$ is an NP-hard problem. In the case of $q \geq 1$, the following $\ell_1$ relaxation 
\begin{equation}
\label{eq:L2L1-B}
P'_{\ell_q/\ell_1}: \; \min \sum_{i=1}^{n}{\| \B[i] \c[i] \|_q} \quad \st \quad \y = \B \c,
\end{equation}
is a convex program and can be solved efficiently. The work of  \cite{Ganesh:ICASSP09} studies conditions under which, for the special case of $q=2$, $P'_{\ell_2/\ell_1}$ and $P'_{\ell_2/\ell_0}$ are equivalent. The conditions are based on the notion of mutual subspace incoherence, as described next.

\myparagraph{Mutual Subspace Incoherence} The work of  \cite{Ganesh:ICASSP09} introduces the notion of mutual subspace incoherence of $\B$, which is defined as
\begin{equation}
\mu_S = \max_{i \neq j} \max_{\x \in \mathcal{S}_i, \z \in \mathcal{S}_j} \frac{| \x^{\top} \z |}{ \| \x \|_2  \| \z \|_2 },
\end{equation}
%
where $\S_i = \spn(\B[i])$. Under the assumption that the blocks have  linearly independent columns and the subspaces spanned by each block are disjoint, \cite{Ganesh:ICASSP09} shows that $P'_{\ell_2/\ell_1}$  and $P'_{\ell_2/\ell_0}$ are equivalent if
\begin{equation}
(2k-1) \mu_S < 1.
\end{equation}
%

%

\smallskip
As mentioned above, the state-of-the-art block-sparse recovery methods \cite{Stojnic:TSP09, Eldar:TSP10, Eldar:TIT09, Ganesh:ICASSP09, Boufounos:TIT11} consider dictionaries whose blocks consist of linearly independent vectors which we refer to as \emph{non-redundant blocks}. However, in signal/image processing, machine learning, and computer vision problems such as face recognition \cite{Wright:PAMI09, Elhamifar:CVPR11} and motion segmentation \cite{Elhamifar:CVPR09, Rao:CVPR08}, blocks of a dictionary consist of data points and often the number of data in each block exceeds the dimension of the underlying subspace. 
For example, in automatic face recognition, the number of training images in each block of the dictionary is often more than the dimension of the face subspace, known to be $9$ under a fixed pose and varying illumination \cite{Basri:PAMI03}. One motivation for this is the fact that having more data in each block better captures the underlying distribution of the data in each subspace and, as expected, increases the performance of tasks such as classification. However, to the best of our knowledge, existing theoretical results have not addressed recovery in dictionaries whose blocks have linearly dependent atoms, which we refer to as \emph{redundant blocks}. 
Moreover, theoretical analysis for the equivalence between $P_{\ell_q/\ell_1}$ and $P_{\ell_q/\ell_0}$ as well as the equivalence between $P'_{\ell_q/\ell_1}$ and $P'_{\ell_q/\ell_0}$ has been restricted to only $q=2$. Nevertheless, empirical studies in some applications \cite{Zhao:AnnalsStat09}, have shown better block-sparse recovery performance for $q \neq 2$. Therefore, there is a need for analyzing the performance of each class of the convex programs for arbitrary $q \geq 1$.

\subsection{Paper Contributions}
In this paper, we consider the problem of block-sparse recovery using the two classes of convex programs $P_{\ell_q/\ell_1}$ and $P'_{\ell_q/\ell_1}$ for $q \geq 1$. Unlike the state of the art, we do not restrict the blocks of a dictionary to have linearly independent columns. Instead, we allow for both non-redundant and redundant blocks. In addition, we do not impose any restriction on the lengths of the blocks, such as requiring them to have the same length, and allow arbitrary and different lengths for the blocks. To characterize the relation between blocks of a dictionary, we introduce the notions of mutual and cumulative subspace coherence, which can be thought of as natural extensions of mutual and cumulative coherence from one-dimensional to multi-dimensional subspaces. 
Based on these notions, we derive conditions under which the convex programs $P_{\ell_q/\ell_1}$ and $P'_{\ell_q/\ell_1}$ are equivalent to $P_{\ell_q/\ell_0}$ and $P'_{\ell_q/\ell_0}$, respectively. While the mutual subspace coherence is easier to compute, cumulative subspace coherence provides weaker conditions for block-sparse recovery using either of the convex programs. Thanks to our analysis framework and the introduced notions of subspace coherence, our block-sparse recovery conditions are weaker than the conditions of the state of the art who have studied the special case of $q=2$. To the best of our knowledge, our work is the first one to analyze both non-redundant and redundant blocks, while our theoretical framework does not separate the two cases and analyzes both within a unified framework.

We evaluate the performance of the proposed convex programs on synthetic data and in the problem of face recognition. Our results show that treating the face recognition as a block-sparse recovery problem can significantly improve the recognition performance. In fact, we show that the convex program $P'_{\ell_q/\ell_1}$ outperforms the state-of-the-art face recognition methods on a real-world dataset.

\myparagraph{Paper Organization} The paper is organized as follows. In Section \ref{sec:problemsetting}, we introduce some notations and notions that characterize the relation between blocks and the relation among atoms within each block of a dictionary. 
In Section \ref{sec:uniqueness}, we investigate conditions under which we can uniquely determine a block-sparse representation of a signal. In Sections \ref{sec:L2L1} and \ref{sec:L2L1P}, we consider the convex programs $P_{\ell_q/\ell_1}$ and $P'_{\ell_q/\ell_1}$, respectively, and study conditions under which they recover a block-sparse representation of a signal in the case of both non-redundant and redundant blocks. 
In Section \ref{sec:discussion}, we discuss the connection between our results and the problem of correcting sparse outlying entries in the observed signal. In Section \ref{sec:experiment}, we evaluate the performance of the two classes of convex programs through a number of synthetic experiments as well as the real problem of face recognition. Finally, Section \ref{sec:conc} concludes the paper.

\section{Problem Setting}
\label{sec:problemsetting}
We consider the problem of block-sparse recovery in a union of subspaces. We assume that the dictionary $\B$ consists of $n$ blocks and the vectors in each block $\B[i] \in \Re^{D \times m_i}$ live in a linear subspace $\mathcal{S}_i$ of dimension $d_i$. Unlike the state-of-the-art block-sparse recovery literature, we do not restrict the blocks to have linearly independent columns. Instead, we allow for both non-redundant ($m_i = d_i$) and redundant ($m_i > d_i$) blocks. 
For reasons that will become clear in the subsequent sections, throughout the paper, we assume that the subspaces $\{\mathcal{S}_i\}_{i=1}^n$ spanned by the columns of the blocks $\{ \B[i] \}_{i=1}^n$ are disjoint.
\vspace{1mm}
\begin{definition}
\label{def:disjoint}
A collection of subspaces $\{\mathcal{S}_i\}_{i=1}^n$ is called \emph{disjoint} if each pair of different subspaces intersect only at the origin.
\end{definition}
\vspace{1mm}

In order to characterize a dictionary $\B$, we introduce two notions that characterize the relationship between the blocks and among the atoms of each block of the dictionary. We start by introducing notions that capture the inter-block relationships of a dictionary. To do so, we make use of the subspaces associated with the blocks. 
\vspace{1mm}
\begin{definition}
\label{def:subspaceangle}
The \emph{subspace coherence} between two disjoint subspaces $\mathcal{S}_i$ and $\mathcal{S}_j$ is defined as
\begin{equation}
\label{eq:sub-coh}
\mu(\mathcal{S}_i,\mathcal{S}_j) = \max_{\x \in \mathcal{S}_i, \z \in \mathcal{S}_j} \frac{| \x^{\top} \z |}{ \| \x \|_2  \| \z \|_2 }  \in [0,1).
\end{equation}
The \emph{mutual subspace coherence} \cite{Ganesh:ICASSP09}, $\mu_S$, is defined as the largest subspace coherence among all pairs of subspaces,
\begin{equation}
\label{eq:mut-sub-coh}
\mu_S \triangleq \max_{i \neq j} \mu(\mathcal{S}_i,\mathcal{S}_j).
\end{equation}
\end{definition}
\vspace{1mm}
%
%
Notice from Definition \ref{def:disjoint} that two disjoint subspaces intersect only at the origin. Therefore, their subspace coherence is always smaller than one.\footnote{Note that the smallest principal angle \cite{Golub:MatrixComps96} between $\mathcal{S}_i$ and $\mathcal{S}_j$, $\theta(\mathcal{S}_i,\mathcal{S}_j)$, is related to the subspace coherence by $\mu(\mathcal{S}_i,\mathcal{S}_j) = \cos(\theta(\mathcal{S}_i,\mathcal{S}_j))$. Thus, $\mu_S$ is the cosine of the smallest principal angle among all pairs of different subspaces.} 
%
The following result shows how to compute the subspace coherence efficiently from the singular values of a matrix obtained from the subspace bases \cite{Golub:MatrixComps96}.
\vspace{1mm}
\begin{proposition}
\label{prop:coherence}
Let $\mathcal{S}_i$ and $\mathcal{S}_j$ be two disjoint subspaces with \emph{orthonormal bases} $\A_i$ and $\A_j$, respectively. The subspace coherence $\mu(\mathcal{S}_i,\mathcal{S}_j)$ is given by
\begin{equation} 
\mu(\mathcal{S}_i,\mathcal{S}_j) = \sigma_1(\A_i^{\top} \A_j).
\end{equation}
\end{proposition}
\vspace{1mm}
It follows from Definition \ref{def:subspaceangle} that the mutual subspace coherence can be computed as
\begin{equation}
\label{eq:mut-coh-svd}
\mu_S = \max_{i \neq j} \sigma_1(\A_i^{\top} \A_j).
\end{equation}
Comparing this with the notion of block-coherence in \eqref{eq:block-coh}, the main difference is that block-coherence, $\mu_B$, uses directly block matrices which are assumed to be non-redundant. However, mutual subspace coherence, $\mu_S$, uses orthonormal bases of the blocks that can be either non-redundant or redundant. The two notions coincide with each other when the blocks are non-redundant and consist of orthonormal vectors. 

While the mutual subspace coherence can be easily computed, it has the shortcoming of not characterizing very well the collection of subspaces because it only reflects the most extreme correlations between subspaces. Thus, we define a notion that better characterizes the relationship between the blocks of a dictionary.
%
%
%
\vspace{1mm}
\begin{definition}
\label{def:cumsubang}
Let $\Lambda_k$ denote a subset of $k$ different elements from $\{1,\ldots, n\}$. The \emph{$k$-cumulative subspace coherence} is defined as
\begin{equation}
\label{eq:cumsubang}
\zeta_k \triangleq \max_{\Lambda_k} \max_{i \notin \Lambda_k} {\sum_{j \in \Lambda_k} \mu(\mathcal{S}_i,\mathcal{S}_j )}.
\end{equation}
\end{definition}
\vspace{1mm}
Roughly speaking, the $k$-cumulative subspace coherence measures the maximum total subspace coherence between a fixed subspace and a collection of $k$ other subspaces. Note that for $k=1$, we have $\zeta_1 = \mu_S$.

%
Mutual/cumulative subspace coherence can be thought of as natural extensions of mutual/cumulative coherence, defined in \eqref{eq:mutcoh} and \eqref{eq:cumcoh}. 
In fact, they are equivalent to each other for the case of one-dimensional subspaces, where each block of the dictionary consists of a single atom. 
%
The following Lemma shows the relationship between mutual and cumulative subspace coherence of a dictionary.
\vspace{1mm}
\begin{lemma} 
\label{lem:cum-mut}
Consider a dictionary $\B$, which consists of $n$ blocks. For every $k \leq n$, we have
\begin{equation}
\zeta_k \leq k \mu_S.
\end{equation}
\end{lemma}
\vspace{1mm}
The proof of Lemma \ref{lem:cum-mut} is straightforward and is provided in the Appendix. 
While computing $\zeta_k$ is, in general, more costly than computing $\mu_S$, it follows from Lemma \ref{lem:cum-mut} that conditions for block-sparse recovery based on $\zeta_k$ are weaker than those based on $\mu_S$, as we will show in the next sections. In fact, for a dictionary, $\zeta_k$ can be much smaller than $k \mu_S$, which results in weaker block-sparse recovery conditions based on $\zeta_k$. To see this, consider the four one-dimensional subspaces shown in Figure \ref{fig:4-subspaces}, where $\mathcal{S}_1$ and $\mathcal{S}_2$ are orthogonal to $\mathcal{S}_3$ and $\mathcal{S}_4$, respectively. Also, the principal angles between $\mathcal{S}_1$ and $\mathcal{S}_2$ as well as $\mathcal{S}_3$ and $\mathcal{S}_4$ are equal to $\theta < \pi / 4$. 
Hence, the ordered subspace coherences are 
$0 \leq 0 \leq \sin(\theta) \leq \sin(\theta) \leq \cos(\theta) \leq \cos(\theta)$. One can verify that
\begin{equation*}
\zeta_3 = \cos(\theta)  +  \sin(\theta) <  3 \mu_S = 3\cos(\theta).
\end{equation*}
In fact, for small values of $\theta$, $\zeta_3$ is much smaller than $3 \mu_S$.{\footnote{Another notion, which can be computed efficiently, is the sum of the $k$ largest subspace coherences, $u_k \triangleq \mu_1 + \cdots + \mu_k$, where the sorted subspace coherences among all pairs of different subspaces are denoted by $\mu_S = \mu_{1} \geq \mu_{2} \geq \mu_{3} \geq \cdots.$ We can show that $\zeta_k \leq u_k \leq k \mu_S$. In the example of Figure \ref{fig:4-subspaces}, $u_3 = 2 \cos(\theta) + \sin(\theta)$, which is between $\zeta_3$ and $3 \mu_S$.}
\begin{figure}
\centering 
\includegraphics[scale = 0.6, trim = 14 13 4 3 , clip]{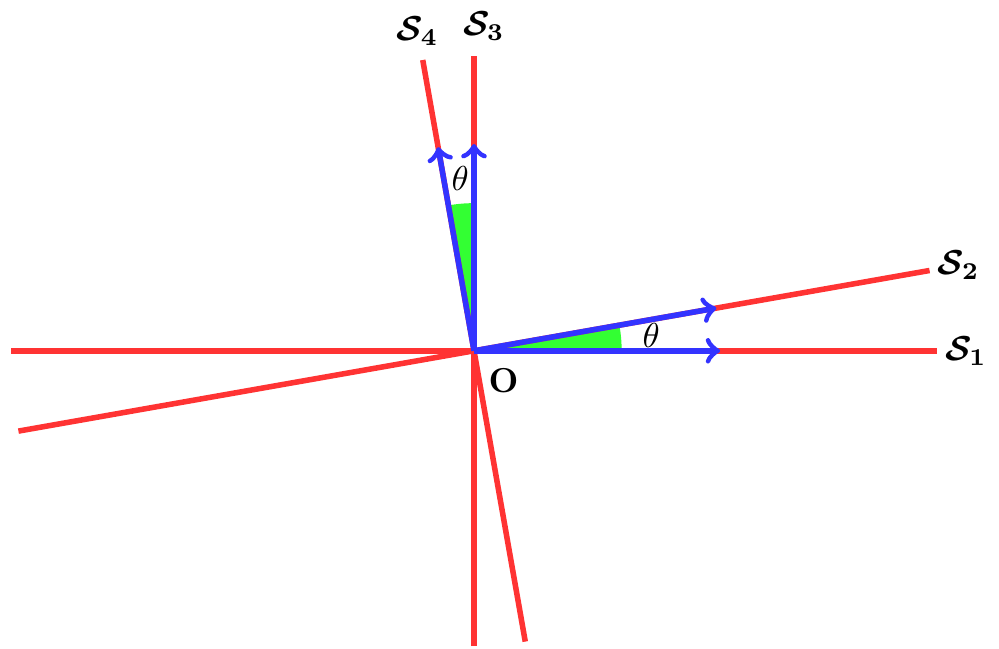}
\vspace{-2mm}
\caption{\footnotesize{Four one-dimensional subspaces in a two-dimensional space. $\mathcal{S}_1$ and $\mathcal{S}_2$ are orthogonal to $\mathcal{S}_3$ and $\mathcal{S}_4$, respectively.}}
\label{fig:4-subspaces}
\end{figure}

Next, we introduce notions that capture the intra-block characteristics of a dictionary.
\vspace{1mm}
\begin{definition}
\label{def:genRIP}
Let $q > 0$. For a dictionary $\B$, define the \emph{intra-block $q$-restricted isometry constant}, $\epsilon_{q}$, as the smallest constant such that for every $i$ there exists a full column-rank submatrix $\bar{\B}[i] \in \Re^{D \times d_i}$ of $\B[i] \in \Re^{D \times m_i}$ such that for every $\bar{\c}[i]$ we have
\begin{equation}
\label{eq:delta-i}
( 1 - \epsilon_{q} ) \| \bar{\c}[i] \|_q^2 \leq \| \bar{\B}[i] \bar{\c}[i] \|_2^2 \leq ( 1 + \epsilon_{q} ) \| \bar{\c}[i] \|_q^2.
\end{equation}
\end{definition}
\vspace{1mm}
%
%
Roughly speaking, $\epsilon_{q}$ characterizes the best $q$-restricted isometry property among all submatrices of $\B[i]$ that span subspace $\S_i$. When $q = 2$, for a dictionary with non-redundant blocks, where $\bar{\B}[i] = \B[i]$, $\epsilon_2$ coincides with the $1$-block restricted isometry constant of $\B$ defined in \eqref{eq:block-RIP}, \ie, $\epsilon_2 = \delta_{B,1}$. Thus, $\epsilon_q$ can be thought of as a generalization of the $1$-block restricted isometry constant, $\delta_{B,1}$, to generic dictionaries with both non-redundant and redundant blocks and arbitrary $q \geq 1$.

\vspace{1mm}
\begin{definition}
\label{def:genRIP2}
Let $q > 0$. For a dictionary $\B$, define the \emph{upper intra-block $q$-restricted isometry constant}, $\sigma_q$, as the smallest constant such that for every $i$ and $\c[i]$ we have
\begin{equation}
\label{eq:Delta}
\| \B[i] \c[i] \|_2^2 \leq (1+\sigma_q) \| \c[i] \|_q^2.
\end{equation}
\end{definition}
\vspace{1mm}
%
%
While in general $\epsilon_q \leq \sigma_q$, for the special case of non-redundant blocks, where $\bar{\B}[i] = \B[i]$, we have $\epsilon_q = \sigma_q$. 
\vspace{1mm}
\begin{remark}
It is important to note that the theory developed in this paper holds for any $q > 0$. However, as we will see, $q \geq 1$ leads to convex programs that can be solved efficiently. Hence, we focus our attention to this case.
\end{remark}

\section{Uniqueness of Block-Sparse Representation}
\label{sec:uniqueness} 
Consider a dictionary $\B$ with $n$ blocks $\B[i] \in \Re^{D \times m_i}$ generated by disjoint subspaces $\S_i$ of dimensions $d_i$. Let $\y$ be a signal that has a block-sparse representation in $\B$ using $k$ blocks indexed by $\{i_1, \ldots, i_k\}$. We can write
\begin{equation}
\y = \sum_{l=1}^{k}{\B[i_l] \c[i_l]} = \sum_{l=1}^{k}{\s_{i_l}},
\end{equation}
where $\s_{i_l} \triangleq \B[i_l] \c[i_l]$ is a vector in the subspace $\S_{i_l}$. In this section, we investigate conditions under which we can uniquely recover the indices $\{i_l\}$ of the blocks/subspaces as well as the vectors $\{ \s_{i_l} \in \S_{i_l} \}$ that generate a block-sparse representation of a given $\y$. We will investigate the efficient recovery of such a block-sparse representation using the convex programs $P_{\ell_q/\ell_1}$ and $P'_{\ell_q/\ell_1}$ in the subsequent sections.

In general, uniqueness of $\{\s_{i_l}\}$ is a weaker notion than the uniqueness of $\{\c[i_l]\}$ since a unique set of coefficient blocks $\{ \c[i_l] \}$ uniquely determines the vectors $\{ \s_{i_l} \}$, but the converse is not necessarily true. More precisely, given $\s_{i_l}$, the equation $\s_{i_l} = \B[i_l] \c[i_l]$ does not have a unique solution $\c[i_l]$ when $\B[i_l]$ is redundant. The solution is unique, only when the block is non-redundant. Therefore, the uniqueness conditions we present next, are more general than the state-of-the-art results. While \cite{Eldar:TIT09} and \cite{Eldar:TSP10} provide conditions for the uniqueness of the blocks $\{ i_l \}$ and the coefficient blocks $\{ \c[i_l] \}$, which only hold for non-redundant blocks, we provide conditions for the uniqueness of the blocks $\{ i_l \}$ and the vectors $\{ \s_{i_l} \}$ for generic dictionaries with non-redundant or redundant blocks. We show the following result whose proof is provided in the Appendix.
\vspace{1mm}
\begin{proposition}
\label{prop:uniqueness1}
Let $\bar{\B}[i] \in \Re^{D \times d_i}$ be an arbitrary full column-rank submatrix of $\B[i] \in \Re^{D \times m_i}$ and define
\begin{equation}
\label{eq:Bbar}
\bar{\B} \triangleq \begin{bmatrix} \bar{\B}[1] & \cdots & \bar{\B}[n] \end{bmatrix}.
\end{equation}
%
The blocks $\{ i_l \}$ and the vectors $\{ \s_{i_l} \}$ that generate a $k$-block-sparse representation of a signal can be determined uniquely if and only if $\bar{\B} \, \bar{\c} \neq 0$ for every $2k$-block-sparse vector $\bar{\c} \neq 0$. 
\end{proposition}
\vspace{1mm}
\begin{remark}
Note that the disjointness of subspaces is a necessary condition for uniquely recovering the blocks that take part in a block-sparse representation of a signal. This comes from the fact that for $k=1$, the uniqueness condition of Proposition \ref{prop:uniqueness1} requires that any two subspaces intersect only at the origin. 
\end{remark}
\vspace{1mm}

Next, we state another uniqueness result 
that we will use in our theoretical analysis in the next sections. 
For a fixed $\tau \in [0,1)$ and for each $i \in \{1,\ldots, n\}$ define
\begin{equation}
\label{eq:Wi}
\mathbb{W}_{\tau, i} \triangleq \{ \s_i \in \mathcal{S}_i,\, 1-\tau \leq \| \s_{i} \|_2^2 \leq 1+\tau \},
\end{equation}
which is the set of all vectors in $\mathcal{S}_i$ whose norm is bounded by $1+\tau$ from above and by $1-\tau$ from below.
Let $\Lambda_k = \{ i_1, \ldots, i_k \}$ be a set of $k$ indices from $\{1, \ldots, n \}$. Define
%
\begin{equation}
\label{eq:mathbbB}
\mathbb{B}_{\tau}(\Lambda_k) \! \triangleq \! \{ \B_{\Lambda_k} \!=\! \begin{bmatrix} \s_{i_1} \!\!\!\!\!&\cdots  \!\!\!\!\!& \s_{i_k} \end{bmatrix}\!, \s_{i_l} \! \in \! \mathbb{W}_{\tau, i_l}, 1 \leq l \leq k \},
\end{equation}
which is the set of matrices $\B_{\Lambda_k} \in \Re^{D \times k}$ whose columns are drawn from subspaces indexed by $\Lambda_k$ and their norms are bounded according to \eqref{eq:Wi}. With abuse of notation, we use $\B_k$ to indicate $\B_{\Lambda_k} \in \mathbb{B}_{\tau}(\Lambda_k)$ whenever $\Lambda_k$ is clear from the context. For example, $\B_n \in \Re^{D \times n}$ indicates a matrix whose columns are drawn from all $n$ subspaces.
%
%
We have the following result whose proof is provided in the Appendix. 
\vspace{1mm}
\begin{corollary}
\label{cor:uniqueness1}
Let $\tau \in [0,1)$. 
The blocks $\{ i_l \}$ and the vectors $\{ \s_{i_l} \}$ that constitute a $k$-block-sparse representation of a signal can be determined uniquely if and only if $\rank(\B_{n}) \geq 2k$ for every $\B_{n} \in \mathbb{B}_{\tau}(\Lambda_n)$.
\end{corollary}
\vspace{1mm}
Note that the result of Corollary \ref{cor:uniqueness1} still holds if 
we let the columns of $\B_n$ have arbitrary nonzero norms, because the rank of a matrix does not change by scaling its columns with nonzero constants. However, as we will show in the next section, the bounds on the norms as in \eqref{eq:Wi} appear when we analyze block-sparse recovery using the convex programs $P_{\ell_q/\ell_1}$ and $P'_{\ell_q/\ell_1}$. While checking the condition of Corollary \ref{cor:uniqueness1} is not possible, as it requires computing every possible $\s_i$ in $\mathbb{W}_{\tau, i}$, we use the result of Corollary \ref{cor:uniqueness1} in our theoretical analysis in the next sections.
In the remainder of the paper, we assume that a given signal $\y$ has a \emph{unique} $k$-block-sparse representation in $\B$. By {uniqueness of a block-sparse representation}, we mean that the blocks $\Lambda_k$ and the vectors $\{ \s_i \in \S_i \}_{i \in \Lambda_k}$ for which $\y = \sum_{i \in \Lambda_k}{\s_i}$ can be determined uniquely. Under this assumption, we investigate conditions under which the convex programs  $P_{\ell_q/\ell_1}$ and $P'_{\ell_q/\ell_1}$ recover the unique set of nonzero blocks $\Lambda_k$ and the unique vectors $\{ \s_i \in \S_i \}_{i \in \Lambda_k}$.

\section{Block-Sparse Recovery via $P_{\ell_q/\ell_1}$}
\label{sec:L2L1}
%
The problem of finding a block-sparse representation of a signal $\y$ in a dictionary $\B$ with non-redundant or redundant blocks can be cast as the optimization program
\begin{equation*}
\label{eq:optL2L0G}
P_{\ell_q/\ell_0}: \; \min \sum_{i=1}^{n} I(\| \c[i] \|_q) \quad \st \quad \y = \B \c.
\end{equation*}
Since $P_{\ell_q/\ell_0}$ directly penalizes the norm of the coefficient blocks, it always recovers a representation of a signal with the minimum number of nonzero blocks. 
As $P_{\ell_q/\ell_0}$ is NP-hard, for $q \geq 1$, we consider its convex relaxation
%
\begin{equation*}
\label{eq:optL2L1G}
P_{\ell_q/\ell_1}: \; \min \sum_{i=1}^{n} \| \c[i] \|_q \quad \st \quad \y = \B \c,
\end{equation*}
and, under a unified framework, study conditions under which $P_{\ell_q/\ell_1}$ and $P_{\ell_q/\ell_0}$ are equivalent for both the case of non-redundant and redundant blocks.

To that end, let $\Lambda_k$ be a set of $k$ indices from $\{1, \cdots, n\}$ and $\Lambda_{\widehat{k}}$ be the set of the remaining $n-k$ indices. Let $\x$ be a nonzero vector in the intersection of $\oplus_{i \in \Lambda_k}{\mathcal{S}_i}$ and $\oplus_{i \in \Lambda_{\widehat{k}}}{\mathcal{S}_i}$, where $\oplus$ denotes the direct sum operator. Let the minimum $\ell_q/\ell_1$-norm coefficient vector when we choose only the $k$ blocks of $\B$ indexed by $\Lambda_k$ be
\begin{equation}
\label{eq:L2L1red1}
\breve{\c}^* = \argmin \sum_{i \in \Lambda_k} \| \c[i] \|_q \quad \st \quad \x = \! \sum_{i \in \Lambda_k} \B[i] \c[i],
\end{equation}
and let the minimum $\ell_q/\ell_1$-norm coefficient vector when we choose the blocks indexed by $\Lambda_{\widehat{k}}$ be
\begin{equation}
\label{eq:L2L1red2}
\widehat{\c}^{*} = \argmin \sum_{i \in \Lambda_{\widehat{k}}} \| \c[i] \|_q \quad \st \quad \x = \! \sum_{i \in \Lambda_{\widehat{k}}} \B[i] \c[i].
\end{equation}
%
%
The following theorem gives conditions under which the convex program $P_{\ell_q/\ell_1}$ is guaranteed to successfully recover a $k$-block-sparse representation of a given signal.
\vspace{1mm}
\begin{theorem}
\label{thm:SuffRedundant} 
For all signals that have a unique $k$-block-sparse representation in $\B$, the solution of the optimization program $P_{\ell_q/\ell_1}$ is equivalent to that of $P_{\ell_q/\ell_0}$, if and only if 
\begin{multline}
\label{eq:SuffRedundant}
\forall \Lambda_k, \, \forall \x \in (\oplus_{i \in \Lambda_k}\mathcal{S}_i) \; \cap \; (\oplus_{i \in \Lambda_{\widehat{k}}}\mathcal{S}_i), \, \x \neq \0  \\ \implies \sum_{i \in \Lambda_k} \| \breve{\c}^*[i] \|_q < \sum_{i \in \Lambda_{\widehat{k}}} \| \widehat{\c}^*[i] \|_q.
\end{multline}
\end{theorem}
%
%
%
\begin{proof}
$(\Longleftarrow)$ Fix $\Lambda_k$ and $\y$ in $\oplus_{i \in \Lambda_k}{\mathcal{S}_i}$ and let $\c^*$ be the solution of $P_{\ell_q/\ell_1}$. If $\c^*$ has at most $k$ nonzero blocks, then by the uniqueness assumption the nonzero blocks are indexed by $\Lambda_k$. For the sake of contradiction, assume that $\c^*$ has more than $k$ nonzero blocks, so $\c^*$ is nonzero for some blocks in $\Lambda_{\widehat{k}}$. Define 
\begin{equation}
\label{eq:ybar-red}
\x \triangleq \y - \sum_{i \in \Lambda_k} \B[i] \c^*[i] = \sum_{i \in \Lambda_{\widehat{k}}} \B[i] \c^*[i].
\end{equation}
From \eqref{eq:ybar-red} we have that $\x$ lives in the intersection of $\oplus_{i \in \Lambda_k} \mathcal{S}_i$ and $\oplus_{i \in \Lambda_{\widehat{k}}} \mathcal{S}_i$. Let $\breve{\c}^*$ and $\widehat{\c}^*$ be respectively the solutions of the optimization problems in \eqref{eq:L2L1red1} and \eqref{eq:L2L1red2}, for $\x$. We can write
\begin{equation}
\label{eq:ybar-red2}
\x = \sum_{i \in \Lambda_k} \B[i] \breve{\c}^*[i] = \sum_{i \in \Lambda_{\widehat{k}}} \B[i] \widehat{\c}^*[i].
\end{equation}
We also have the following inequalities
\begin{equation}
\label{eq:medsuf}
\sum_{i \in \Lambda_k}{\| \breve{\c}^*[i] \|_q} < \sum_{i \in \Lambda_{\widehat{k}}}{\| \widehat{\c}^*[i] \|_q} \leq \sum_{i \in \Lambda_{\widehat{k}}}{\| \c^*[i] \|_q },\end{equation}
where the first inequality follows from the sufficient condition in \eqref{eq:SuffRedundant}. The second inequality follows from the second inequalities in \eqref{eq:ybar-red} and \eqref{eq:ybar-red2} and the fact that $\widehat{\c}^*$ is the optimal solution of \eqref{eq:L2L1red2} for $\x$.
Using the first equalities in \eqref{eq:ybar-red} and \eqref{eq:ybar-red2}, we can rewrite $\y$ as 
\begin{equation}
\y = \sum_{i \in \Lambda_k} \B[i] ( \c^*[i] + \breve{\c}^*[i] ),
\end{equation}
which implies that $\c^* + \breve{\c}^*$ is a solution of $\y = \B \c$. 
Finally, using \eqref{eq:medsuf} and the triangle inequality, we obtain
%
%
\begin{multline}
\label{eq:trianglesuff}
\sum_{i \in \Lambda_k} \! \| \c^*[i] + \breve{\c}^*[i] \|_q \leq 
\sum_{i \in \Lambda_k} \! \| \c^*[i] \|_q + \sum_{i \in \Lambda_k} \! \| \breve{\c}^*[i] \|_q  \\< 
\sum_{i \in \Lambda_k} \! \| \c^*[i] \|_q + \sum_{i \in \Lambda_{\widehat{k}}} \! \| \widehat{\c}^*[i] \|_q  \leq 
\sum_{i =1}^{n} \! {\| \c^*[i] \|_q }.
\end{multline}
This contradicts the optimality of $\c^*$, since it means that $\c^* + \breve{\c}^*$, which is also a solution of $\y = \B \c$, has a strictly smaller $\ell_q/\ell_1$-norm than $\c^*$.

\noindent $(\Longrightarrow)$ We prove this using contradiction. Assume there exist $\Lambda_k$ and $\x$ in the intersection of $\oplus_{i \in \Lambda_k}{\mathcal{S}_i}$ and $\oplus_{i \in \Lambda_{\widehat{k}}}{\mathcal{S}_i}$ for which the condition in \eqref{eq:SuffRedundant} does not hold, \ie,
\begin{equation}
\sum_{i \in \Lambda_{\widehat{k}}} \| \widehat{\c}^*[i] \|_q \leq \sum_{i \in \Lambda_k} \| \breve{\c}^*[i] \|_q.
\end{equation}
Thus, a solution of $\x = \B \c$ is given by $\widehat{\c}^*$ that is not $k$-block-sparse and has a $\ell_q/\ell_1$-norm that is smaller  than or equal to any $k$-block-sparse solution, contradicting the equivalence of $P_{\ell_q/\ell_1}$ and $P_{\ell_q/\ell_0}$.
\end{proof}
\vspace{1mm}
The condition of Theorem \ref{thm:SuffRedundant} (and Theorem \ref{thm:SuffRedundant2} in the next section) is closely related to the nullspace property in \cite{VandenBerg:TIT10, Lai:ACHA11, Stojnic:TSP09, Donoho:DCG06}. However, the key difference is that we do not require the condition of Theorem \ref{thm:SuffRedundant} to hold for all feasible vectors of \eqref{eq:L2L1red1} and \eqref{eq:L2L1red2}, denoted by $\breve{\c}$ and $\widehat{\c}$, respectively. Instead, we only require the condition of Theorem \ref{thm:SuffRedundant} to hold for the optimal solutions of \eqref{eq:L2L1red1} and \eqref{eq:L2L1red2}. Thus, while the nullspace property might be violated by some feasible vectors $\breve{\c}$ and $\widehat{\c}$, our condition can still hold for  $\breve{\c}^*$ and $\widehat{\c}^*$, guaranteeing the equivalence of the two optimization programs. 

Notice that it is not possible to check the condition in \eqref{eq:SuffRedundant} for every $\Lambda_k$ and for every $\x$ in the intersection of $\oplus_{i \in \Lambda_k}{\mathcal{S}_i}$ and $\oplus_{i \in \Lambda_{\widehat{k}}}{\mathcal{S}_i}$. In addition, the condition in \eqref{eq:SuffRedundant} does not explicitly incorporate the inter-block and intra-block parameters of the dictionary. In what follows, we propose sufficient conditions that incorporate the inter-block and intra-block parameters of the dictionary and can be efficiently checked. We use the following Lemma whose proof is provided in the Appendix. 
\vspace{1mm}
\begin{lemma}
\label{lem:matrix}
Let $\E_k \in \Re^{D \times k}$ be a matrix whose columns are chosen from subspaces indexed by $\Lambda_k$ and $\E_k \in \mathbb{B}_{\alpha}(\Lambda_k)$ for a fixed $\alpha \in [0,1)$. Let $\E_{\widehat{k}} \in \Re^{D \times n-k}$ be a matrix whose columns are chosen from subspaces indexed by $\Lambda_{\widehat{k}}$ where the Euclidean norm of each column is less than or equal to $\sqrt{1+\beta}$. We have
\begin{equation}
\| (\E_k^{\top} \E_k )^{-1} \E_k^{\top} \E_{\widehat{k}} \|_{1,1} \leq \frac{ \sqrt{(1+\alpha) (1+\beta)} \, \zeta_k }{ 1 - [\, \alpha + (1+\alpha) \zeta_{k-1} \,] }.
\end{equation}
\end{lemma}
\vspace{1mm}
%
%
%
%
\begin{proposition}
\label{prop:suff1}
For signals that have a unique $k$-block-sparse representation in $\B$, the solution of the optimization program $P_{\ell_q/\ell_1}$ is equivalent to that of $P_{\ell_q/\ell_0}$, if
%
\begin{equation}
\label{eq:cumsubLqL1suff}
\sqrt{\frac{1+\sigma_q}{1+\epsilon_q}} \; \zeta_k + \zeta_{k-1} < \frac{1-\epsilon_q}{1+\epsilon_q}.
\end{equation}
\end{proposition}
\vspace{1mm}
\begin{proof}
Fix a set $\Lambda_k = \{i_1,\ldots,i_k\}$ of $k$ indices from $\{1, \ldots, n \}$ and let $\Lambda_{\widehat{k}} = \{i_{k+1},\ldots,i_n\}$ denote the set of the remaining indices. Consider a signal $\x$ in the intersection of $\oplus_{i \in \Lambda_k} \mathcal{S}_i$ and $\oplus_{i \in \Lambda_{\widehat{k}}} \mathcal{S}_i$. The structure of the proof is as follows. We show that $\x$ can be written as $\x = \B_{k} \a_{k}$, where for the solution of \eqref{eq:L2L1red1}, we have $\sum_{i \in \Lambda_{k}}{\| \breve{\c}^{*}[i] \|_q} \leq \| \a_{k} \|_1$. Also, we show that for the solution of \eqref{eq:L2L1red2}, one can write $\x = \B_{\widehat{k}} \a_{\widehat{k}}$, where $\| \a_{\widehat{k}} \|_1 = \sum_{i \in \Lambda_{\widehat{k}}}{\| \widehat{\c}^{*}[i] \|_q}$. Under the sufficient condition of the Proposition, we show that $\| \a_k \|_1 < \| \a_{\widehat{k}} \|_1$, implying that the condition of Theorem \ref{thm:SuffRedundant} is satisfied.

To start, let $\widehat{\c}^*$ be the solution of the optimization program in \eqref{eq:L2L1red2}. For every $i \in \Lambda_{\widehat{k}}$, define the vectors $\s_i$ and the scalars $a_i$ as follows. If $\widehat{\c}^*[i] \neq 0$ and $\B[i] \widehat{\c}^*[i] \neq 0$, let 
\begin{equation}
\s_i \triangleq \frac{\B[i] \widehat{\c}^*[i]}{ \| \widehat{\c}^*[i] \|_q }, \quad a_i \triangleq \| \widehat{\c}^*[i] \|_q. 
\end{equation}
Otherwise, let $\s_i$ be an arbitrary vector in $\mathcal{S}_i$ of unit Euclidean norm and $a_i = 0$. We can write
\begin{equation}
\label{eq:op2}
\x = \sum_{i \in \Lambda_{\widehat{k}}} \B[i] \widehat{\c}^*[i] = \B_{\widehat{k}} \a_{\widehat{k}},
\end{equation}
where $\B_{\widehat{k}} \triangleq \begin{bmatrix} \s_{i_{k+1}} \!\!& \cdots \!\!& \s_{i_n} \end{bmatrix}$ and $\a_{\widehat{k}} \triangleq \begin{bmatrix} a_{i_{k+1}} \!\! & \cdots \!\! & a_{i_n} \end{bmatrix}^{\top}$. Note that from Definition \ref{def:genRIP2}, we have $\| \s_{i} \|_2 \leq \sqrt{1 + \sigma_q}$ for every $i \in \Lambda_{\widehat{k}}$. 

%
Let $\bar{\B}[i] \in \Re^{D \times d_i}$ be the submatrix of $\B[i]$ associated with $\epsilon_{q}$ according to Definition \ref{def:genRIP}. Since $\bar{\B}[i]$ spans the subspace $\mathcal{S}_i$, there exists $\bar{\c}[i]$ such that 
\begin{equation}
\x = \sum_{i \in \Lambda_k} \bar{\B}[i] \bar{\c}[i] \triangleq \B_k \a_k,
\end{equation}
where $\B_k \triangleq \begin{bmatrix} \s_{i_1} \!\!& \cdots \!\!& \s_{i_k} \end{bmatrix}$ and $\a_k \triangleq \begin{bmatrix} a_{i_1} \!\! & \cdots \!\! & a_{i_k} \end{bmatrix}^{\top}$. For every $i \in \Lambda_k$ the vectors $\s_i$ and the scalars $a_i$ are defined as 
\begin{equation}
\s_i \triangleq \frac{\bar{\B}[i] \bar{\c}[i]}{ \| \bar{\c}[i] \|_q }, \quad a_i \triangleq \| \bar{\c}[i] \|_q,
\end{equation}
whenever $\bar{\c}[i] \neq 0$ and $\bar{\B}[i] \bar{\c}[i] \neq 0$. Otherwise, we let $\s_i$ be an arbitrary vector in $\mathcal{S}_i$ of unit Euclidean norm and $a_i = 0$. Clearly, $\B_k \in \mathbb{B}_{\epsilon_q}(\Lambda_k)$ is full column-rank using Corollary \ref{cor:uniqueness1} when $\epsilon_q \in [0,1)$. Hence, we have $\, \a_k = (\B_k^{\top} \B_k)^{-1}\B_k^{\top} \x \,$ and consequently,
\begin{equation}
\| \a_k \|_1 = \| (\B_k^{\top} \B_k)^{-1}\B_k^{\top} \x \|_1.
\end{equation}
Substituting $\y$ from \eqref{eq:op2} in the above equation, we obtain
\begin{multline}
\label{eq:interm1}
\| \a_k \|_1 = \| (\B_k^{\top} \B_k)^{-1}\B_k^{\top} \B_{\widehat{k}} \a_{\widehat{k}} \|_1 \\ \leq \| (\B_k^{\top} \B_k)^{-1}\B_k^{\top} \B_{\widehat{k}} \|_{1,1} \| \a_{\widehat{k}} \|_1.
\end{multline}
Using Lemma \ref{lem:matrix} with $\alpha = \epsilon_q$ and $\beta = \sigma_q$, we have
\begin{equation}
\| (\B_k^{\top} \B_k)^{-1} \B_k^{\top} \B_{\widehat{k}} \|_1 \leq \frac{ \sqrt{(1+\epsilon_q) (1+\sigma_q)} \, \zeta_k }{ 1 - [ \, \epsilon_q  + (1+\epsilon_q) \zeta_{k-1} \,] }. 
\end{equation}

Thus, if the right hand side of the above equation is strictly less than one, \ie, if the condition of the proposition is satisfied, then from \eqref{eq:interm1} we have $\| \a_k \|_1 < \| \a_{\widehat{k}} \|_1$. Finally, using the optimality of $\breve{\c}^*$ when we choose the blocks indexed by $\Lambda_k$, we obtain
%
%
\begin{equation}
\sum_{i \in \Lambda_k}{\!\| \breve{\c}^*[i] \|_q} \leq \sum_{i \in \Lambda_k}{\!\| \bar{\c}[i] \|_q} \!=\! \| \a_k \|_1  < \| \a_{\widehat{k}} \|_1 \!=\!\! \sum_{i \in \Lambda_{\widehat{k}}}{\! \| \widehat{\c}^*[i] \|_q},
\end{equation}
which implies that the condition of Theorem \ref{thm:SuffRedundant} is satisfied. Thus, the convex program $P_{\ell_q/\ell_1}$ recovers a $k$-block-sparse representation of a given signal. 
\end{proof}
\smallskip
The following corollary derives stronger, but simpler to check, sufficient conditions for block-sparse recovery using $P_{\ell_q/\ell_1}$.
\vspace{1mm}
\begin{corollary}
\label{cor:suff1}
For signals that have a unique $k$-block-sparse representation in $\B$, the solution of the optimization program $P_{\ell_q/\ell_1}$ is equivalent to that of $P_{\ell_q/\ell_0}$, if \footnote{An intermediate sufficient condition is given by $\sqrt{\frac{1+\sigma_q}{1+\epsilon_q}} \; u_k + u_{k-1} < \frac{1-\epsilon_q}{1+\epsilon_q}$ using the fact that $\zeta_k \leq u_k \leq k \mu_S$.}
%
\begin{equation}
\label{eq:mutsubLqL1suff}
( k \sqrt{\frac{1+\sigma_q}{1+\epsilon_q}} + k-1 ) \mu_S < \frac{1-\epsilon_q}{1+\epsilon_q}.
\end{equation}
\end{corollary}
\vspace{1mm}
\begin{proof}
The result follows from Proposition \ref{prop:suff1} by using the fact that $\zeta_k \leq k \mu_S$ from Lemma \ref{lem:cum-mut}.
\end{proof}
\vspace{1mm}

For non-redundant blocks, we have $\sigma_q = \epsilon_q$. Thus, in this case, for the convex program $P_{\ell_q/\ell_1}$, the block-sparse recovery condition based on the mutual subspace coherence in \eqref{eq:mutsubLqL1suff} reduces to
\begin{equation}
\label{eq:mutsubLqL1suff-nonred}
(2 k - 1) \mu_S < \frac{1-\epsilon_q}{1+\epsilon_q}.
\end{equation}
Also, for non-redundant blocks with $q=2$, $\epsilon_2$ coincides with the $1$-block restricted isometry constant $\delta_{B,1}$ defined in \eqref{eq:block-RIP}. In this case, note that our result in \eqref{eq:mutsubLqL1suff-nonred} is different from the result of \cite{Eldar:TSP10} stated in \eqref{eq:block-coh-cond}. First, the notion of mutual subspace coherence is different from the notion of block coherence because they are defined as the largest singular values of two different matrices. Second, the bound on the right hand side of \eqref{eq:mutsubLqL1suff-nonred} is a function of the best intra-block $q$-restricted isometry constant of $\B$, while the right hand side of \eqref{eq:block-coh-cond} is a function of the maximum mutual-coherence over all blocks of $\B$. 


For non-redundant blocks, the block-sparse recovery condition based on the cumulative subspace coherence in \eqref{eq:cumsubLqL1suff} reduces to
\begin{equation}
\label{eq:cumsubLqL1suff-nonred}
\zeta_k + \zeta_{k-1} < \frac{1-\epsilon_q}{1+\epsilon_q},
\end{equation}
which is always weaker than the condition based on the mutual subspace coherence in \eqref{eq:mutsubLqL1suff-nonred}. In addition, when $q=2$, \eqref{eq:cumsubLqL1suff-nonred} provides a weaker condition than the one in \eqref{eq:block-coh-cond} which is based on the notion of block-coherence \cite{Eldar:TSP10}.

\section{Block-Sparse Recovery via $P'_{\ell_q/\ell_1}$}
\label{sec:L2L1P}

In this section, we consider the problem of block-sparse recovery from the non-convex optimization program
\begin{equation*}
\label{eq:optL2L0M}
P'_{\ell_q/\ell_0}: \; \min \sum_{i=1}^{n} I(\| \B[i] \c[i] \|_q) \quad \st \quad \y = \B \c.
\end{equation*}
Unlike $P_{\ell_q/\ell_0}$ that penalizes the norm of the coefficient blocks, $P'_{\ell_q/\ell_0}$ penalizes the norm of the reconstructed vectors from the blocks. Hence, $P'_{\ell_q/\ell_0}$ finds a solution that has the minimum number of nonzero vectors $\B[i] \c[i] \in \S_i$. For a dictionary with non-redundant blocks, the solution of $P'_{\ell_q/\ell_0}$ has also the minimum number of nonzero coefficient blocks. However, this does not necessarily hold for a dictionary with redundant blocks since a nonzero $\c[i]$ in the nullspace of the non-contributing blocks ($\B[i] \c[i] = \0$) does not affect either the value of the cost function or the equality constraint in $P'_{\ell_q/\ell_0}$.

Despite the above argument, we consider $P'_{\ell_q/\ell_0}$ and its convex relaxation for block-sparse recovery in generic dictionaries with non-redundant or redundant blocks, because one can simply set to zero the nonzero blocks $\c^*[i]$ for which $\B[i] \c^*[i] $ is zero.\footnote{For noisy data, this can be modified by setting to zero the blocks $\c^*[i]$ for which $\| \B[i] \c^*[i] \|_q$ is smaller than a threshold. In practice, to prevent overfitting, one has to add a small regularization on the coefficients to the optimization program.} Another reason, that we explain in more details in Section \ref{sec:experiment}, comes from the fact that in some tasks such as classification, we are mainly concerned with finding the contributing blocks rather than being concerned with the representation itself. 

Since $P'_{\ell_q/\ell_0}$ is NP-hard, we consider its $\ell_1$ relaxation
\begin{equation*}
\label{eq:optL2L1M}
P'_{\ell_q/\ell_1}: \; \min \sum_{i=1}^{n} \| \B[i] \c[i] \|_q \quad \st \quad \y = \B \c,
\end{equation*}
which is a convex program for $q \geq 1$. 
Our approach to guarantee the equivalence of $P'_{\ell_q/\ell_1}$ and $P'_{\ell_q/\ell_0}$ is similar to the one in the previous section. Let $\Lambda_k$ be a set of $k$ indices from $\{1, \ldots, n \}$ and $\Lambda_{\widehat{k}}$ be the set of the remaining indices. For a nonzero signal $\x$ in the intersection of $\oplus_{i \in \Lambda_k}{\mathcal{S}_i}$ and $\oplus_{i \in \Lambda_{\widehat{k}}}{\mathcal{S}_i}$, let the minimum $\ell_q/\ell_1$-norm coefficient vector when we choose only the blocks of $\B$ indexed by $\Lambda_k$ be
\begin{equation}
\label{eq:L2L1red3}
\breve{\c}^* = \argmin \sum_{i \in \Lambda_k} \! \| \B[i] \c[i] \|_q~~\st~~\x = \sum_{i \in \Lambda_k} \B[i] \c[i].
\end{equation}
Also, let the minimum $\ell_q/\ell_1$-norm coefficient vector when we choose the blocks of $\B$ indexed by $\Lambda_{\widehat{k}}$ be
\begin{equation}
\label{eq:L2L1red4}
\widehat{\c}^{*} = \argmin \sum_{i \in \Lambda_{\widehat{k}}} \! \| \B[i] \c[i] \|_q~~\st~~\x \!=\! \sum_{i \in \Lambda_{\widehat{k}}} \B[i] \c[i].
\end{equation}
%
We have the following result.
\vspace{1mm}
\begin{theorem}
\label{thm:SuffRedundant2}
For all signals that have a unique $k$-block-sparse representation in $\B$, the solution of the optimization program $P'_{\ell_q/\ell_1}$ is equivalent to that of $P'_{\ell_q/\ell_0}$, if and only if  
%
\begin{multline}
\label{eq:SuffRedundant1}
\forall \Lambda_k, \forall \x  \in (\oplus_{i \in \Lambda_k} \mathcal{S}_i) \; \cap \; (\oplus_{i \in \Lambda_{\widehat{k}}} \mathcal{S}_i), \, \x \neq \0 \\ \implies \sum_{i \in \Lambda_k} \| \B[i] \breve{\c}^*[i] \|_q \!<\! \sum_{i \in \Lambda_{\widehat{k}}} \| \B[i] \widehat{\c}^*[i] \|_q.
\end{multline}
\end{theorem}
%
%
\begin{proof}
$(\Longleftarrow)$ Let $\y$ be a signal that lives in the subspace $\oplus_{i \in \Lambda_k}{\mathcal{S}_i}$. Denote by $\c^*$ the solution of the optimization program $P'_{\ell_q/\ell_1}$. 
If for at most $k$ blocks of $\c^*$ we have $\B[i] \c^*[i] \neq \0$, then by the uniqueness assumption, these blocks of $\c^*$ are indexed by $\Lambda_k$. For the sake of contradiction, assume that $\B[i] \c^*[i] \neq \0$ for some $i \in \Lambda_{\widehat{k}}$. Define 
\begin{equation}
\x \triangleq \y - \sum_{i \in \Lambda_k} \B[i] \c^*[i]  =  \sum_{i \in \Lambda_{\widehat{k}}} \B[i] \c^*[i]. 
\end{equation}
The remaining steps of the proof are analogous to the proof of Theorem \ref{thm:SuffRedundant} except that we replace $\| \c^*[i] \|_q$ by $\| \B[i] \c^*[i] \|_q$ in \eqref{eq:medsuf} and use the triangle inequality for $\| \B[i] (\c^*[i]+\breve{\c}^*[i]) \|_q$ in \eqref{eq:trianglesuff}.

\noindent $(\Longrightarrow)$ We prove this using contradiction. Assume that there exist $\Lambda_k$ and $\x$ in the intersection of $\oplus_{i \in \Lambda_k}{\mathcal{S}_i}$ and $\oplus_{i \in \Lambda_{\widehat{k}}}{\mathcal{S}_i}$ for which the condition in \eqref{eq:SuffRedundant1} does not hold, \ie,
\begin{equation}
\sum_{i \in \Lambda_{\widehat{k}}} \| \B[i] \widehat{\c}^*[i] \|_q \leq \sum_{i \in \Lambda_k} \| \B[i] \breve{\c}^*[i] \|_q.
\end{equation}
Thus, a solution of $\x = \B \c$ is given by $\widehat{\c}^*$ that is not $k$-block-sparse and whose linear transformation by $\B$ has a $\ell_q/\ell_1$-norm that is smaller than or equal to the norm of the transformation by $\B$ of any $k$-block-sparse solution, contradicting the equivalence of $P'_{\ell_q/\ell_1}$ and $P'_{\ell_q/\ell_0}$.
\end{proof}
\vspace{1mm}
%

Next, we propose sufficient conditions that incorporate the inter-block and intra-block parameters of the dictionary and can be efficiently checked. Before that we need to introduce the following notation.
%
\vspace{1mm}
\begin{definition}
\label{def:normequiv}
Consider a dictionary $\B$ with blocks $\B[i] \in \Re^{D \times m_i}$. Define $\epsilon'_{q}$ as the smallest constant such that for every $i$ and $\c[i]$ we have
\begin{equation}
\label{eq:delta-i}
( 1 - \epsilon'_{q} ) \| \B[i] \c[i] \|_q^2 \leq \! \| \B[i] \c[i] \|_2^2 \leq \! ( 1 + \epsilon'_{q} ) \| \B[i] \c[i] \|_q^2.
\end{equation}
\end{definition}
\vspace{1mm}
Note that $\epsilon'_q$ characterizes the relation between the $\ell_q$ and $\ell_2$ norms of vectors in $\Re^D$ and does not depend on whether the blocks are non-redundant or redundant. In addition, for $q=2$, we have $\epsilon'_2 = 0$.
\vspace{1mm}
\begin{proposition}
\label{prop:suff2}
For signals that have a unique $k$-block-sparse representation in $\B$, the solution of the optimization program $P'_{\ell_q/\ell_1}$ is equivalent to that of $P'_{\ell_q/\ell_0}$, if 
%
\begin{equation}
\label{eq:cumsubLqL1suff-P}
\zeta_k + \zeta_{k-1} < \frac{1-\epsilon'_q}{1+\epsilon'_q}.
\end{equation}
\end{proposition}
%
%
\begin{proof}
The proof is provided in the Appendix. 
\end{proof}
\smallskip
The following corollary derives stronger yet simpler to check sufficient conditions for block-sparse recovery using $P'_{\ell_q/\ell_1}$.
\vspace{1mm}
\begin{corollary}
\label{cor:suff2}
For signals that have a unique $k$-block-sparse representation in $\B$, the solution of the optimization program $P'_{\ell_q/\ell_1}$ is equivalent to that of $P'_{\ell_q/\ell_0}$, if \footnote{An intermediate sufficient condition is given by $u_k + u_{k-1} < \frac{1-\epsilon'_q}{1+\epsilon'_q}$ using the fact that $\zeta_k \leq u_k \leq k \mu_S$.} 
%
\begin{equation}
\label{eq:mutsubLqL1suff-P}
(2k-1) \mu_S < \frac{1-\epsilon'_q}{1+\epsilon'_q}.
\end{equation}
\end{corollary}
%
%
\begin{proof}
The result follows from Proposition \ref{prop:suff2} by using the fact that $\zeta_k \leq k \mu_S$ from Lemma \ref{lem:cum-mut}.
\end{proof}
\smallskip

Unlike the conditions for the equivalence between $P_{\ell_q/\ell_1}$ and $P_{\ell_q/\ell_0}$, which depend on whether the blocks are non-redundant or redundant, the conditions for the equivalence between $P'_{\ell_q/\ell_1}$ and $P'_{\ell_q/\ell_0}$ do not depend on the redundancy of the blocks. In addition, since $\epsilon'_2 = 0$, the condition for the equivalence between $P'_{\ell_2/\ell_1}$ and $P'_{\ell_2/\ell_0}$ based on the mutual subspace coherence reduces to
\begin{equation}
\label{eq:mutsubLqL1Psuff-nonred}
(2k-1) \mu_S < 1,
\end{equation}
and the condition based on the cumulative subspace coherence reduces to
\begin{equation}
\label{eq:cumsubLqL1Psuff-nonred}
\zeta_{k} + \zeta_{k-1} < 1.
\end{equation}
%
%
%
\begin{remark}
Note that the sufficient conditions in \eqref{eq:mutsubLqL1Psuff-nonred} and \eqref{eq:cumsubLqL1Psuff-nonred} are weaker than the sufficient conditions in \eqref{eq:mutsubLqL1suff-nonred} and \eqref{eq:cumsubLqL1suff-nonred}, respectively. While we can not assert the superiority of $P'_{\ell_2/\ell_1}$ over $P_{\ell_2/\ell_1}$, since the conditions are only sufficient not necessary, as we will show in the experimental results $P'_{\ell_2/\ell_1}$ is in general more successful than $P_{\ell_2/\ell_1}$ for block-sparse recovery.
\end{remark}
\vspace{1mm} 
\begin{remark}
Under the uniqueness assumption, both nonconvex programs $P_{\ell_q/\ell_0}$ and $P'_{\ell_q/\ell_0}$ find the unique blocks $\Lambda_k$ and the vectors $\{ \s_i \in \S_i \}_{i \in \Lambda_k}$ for which $\y = \sum_{i \in \Lambda_k}{\s_i}$. Thus, when the conditions for the success of the convex programs $P_{\ell_q/\ell_1}$ and $P'_{\ell_q/\ell_1}$ hold, their optimal solutions correspond to $\Lambda_k$ and $\{ \s_i \in \S_i \}_{i \in \Lambda_k}$. For non-redundant blocks, this implies that the optimal coefficient vectors found by $P_{\ell_q/\ell_1}$ and $P'_{\ell_q/\ell_1}$ are the same and equal to the true solution. 
\end{remark}

\section{Correcting Sparse Outlying Entries}
\label{sec:discussion}
In real-world problems, observed signals might be corrupted by errors \cite{Wright:PAMI09, Wright:TIT10}, hence might not perfectly lie in the range-space of a few blocks of the dictionary \cite{Elhamifar:CVPR11}. A case of interest, which also happens in practice, is when the observed signal is corrupted with an error that has a few outlying entries. For example, in the face recognition problem, a face image might be corrupted because of occlusions \cite{Wright:PAMI09}, or in the motion segmentation problem, some of the entries of feature trajectories might be corrupted due to objects partially occluding each other or malfunctioning of the tracker \cite{Rao:CVPR08, Elhamifar:CVPR09}. In such cases, the observed signal $\y$ can be modeled as a superposition of a pure signal $\y_0$ and a corruption term $\e$ of the form $\y = \y_0 + \e$, where $\y_0$ has a block sparse representation in the dictionary $\B$ and $\e$ has a few large nonzero entries. 
Thus, $\y$ can be written as 
\begin{equation}
\label{eq:corrupted}
\y = \y_0 + \e = \B \c + \e = \begin{bmatrix} \B & \I \end{bmatrix} \begin{bmatrix} \c \\ \e \end{bmatrix},
\end{equation}
where $\I$ denotes the identity matrix. Note that the new dictionary $\begin{bmatrix} \B & \I \end{bmatrix}$ has still a block structure whose blocks correspond to the blocks of $\B$ and the atoms of $\I$. Thus, in this new dictionary, $\y$ has a block-sparse representation with a few blocks corresponding to $\B$ and a few blocks/atoms corresponding to $\I$. Assuming that the sufficient conditions of the previous sections hold for the dictionary $\begin{bmatrix} \B & \I \end{bmatrix}$, we can recover a block-sparse representation of a corrupted signal using the convex optimization program $P_{\ell_q/\ell_1}$ as
\begin{equation}
P_{\ell_q/\ell_1}\!\!: \; \min \sum_{i=1}^{n}{ \| \c[i] \|_q } +
\| \e \|_1 \;\; \st \;\; \y = \begin{bmatrix} \B & \I
\end{bmatrix} \!\! \begin{bmatrix} \c \\ \e \end{bmatrix}\!,
\end{equation}
or using the convex optimization program $P'_{\ell_q/\ell_1}$ as
\begin{equation}
P'_{\ell_q/\ell_1}\!\!\!:  \, \min \!  \sum_{i=1}^{n}{ \!\| \B[i]
\c[i] \|_q } \!+\! \| \e \|_1 \; \st \; \y \!=\! 
\begin{bmatrix} \B &  \I \end{bmatrix} \!\! \begin{bmatrix} \c \\ \e
\end{bmatrix}\!\!.
\end{equation}
Here, we used the fact that the blocks of $\I$ are of length one, \ie, $\e[i] \in \Re$. Thus, $\sum_{i=1}^{D} \| \e[i] \|_q = \sum_{i=1}^{D} | \e[i] | = \| \e \|_1.$

As a result, this paper not only proposes two classes of convex programs that can be used to deal with block-sparse recovery of corrupted signals, but also provides theoretical guarantees under which one can successfully recover the block-sparse representation of a corrupted signal and eliminate the error from the signal.\footnote{Note that the result can be easily generalized to the case where the error $\e$ has a sparse representation in a dictionary $\G$ instead of $\I$ by considering the dictionary $\begin{bmatrix} \B & \G \end{bmatrix}$ in \eqref{eq:corrupted}.}
\begin{figure*}[t!]
\centering
\includegraphics[width=0.31\linewidth, trim = 34 19 77 10 , clip]{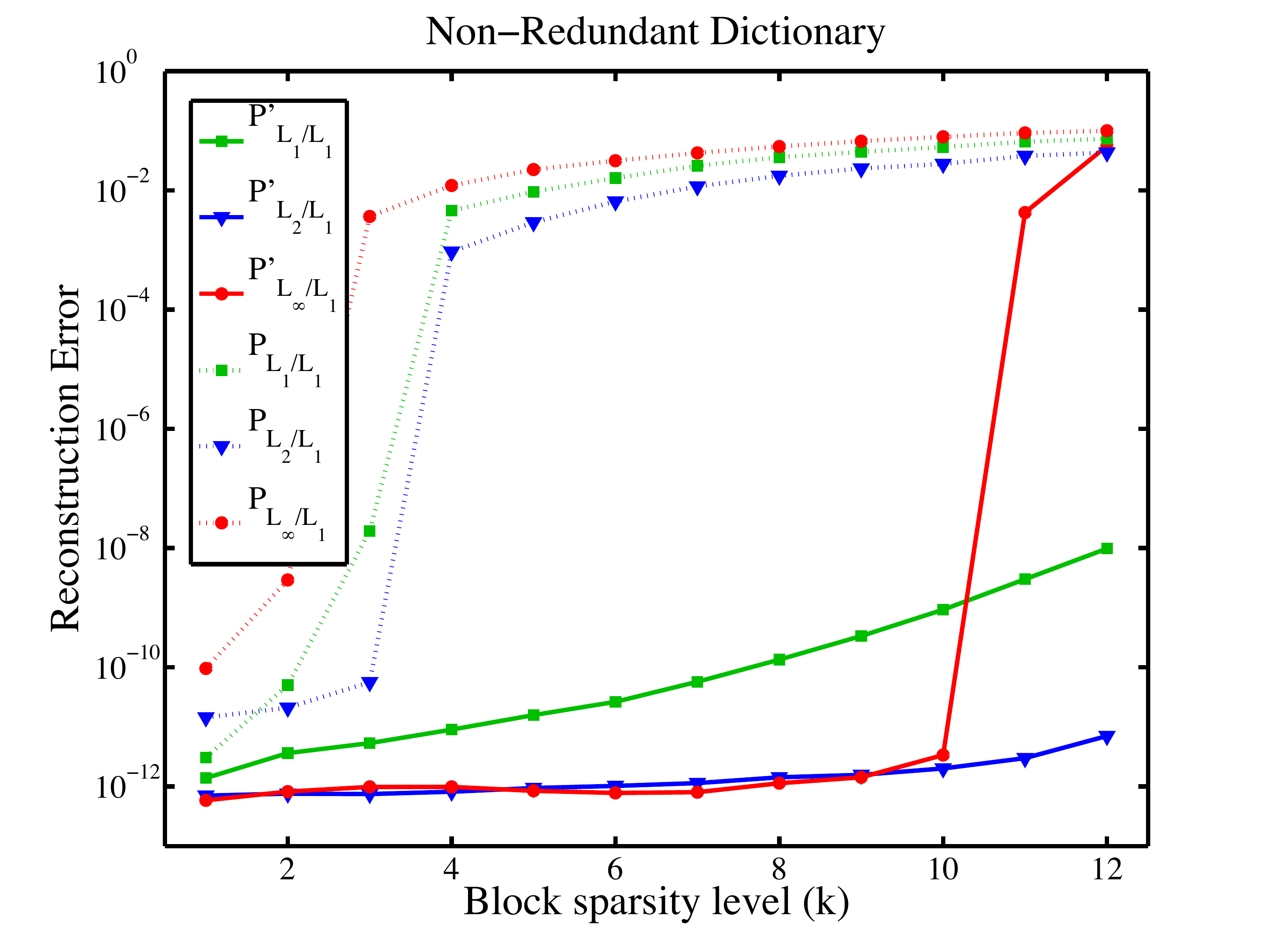} \
\includegraphics[width=0.31\linewidth, trim = 34 19 77 10 , clip]{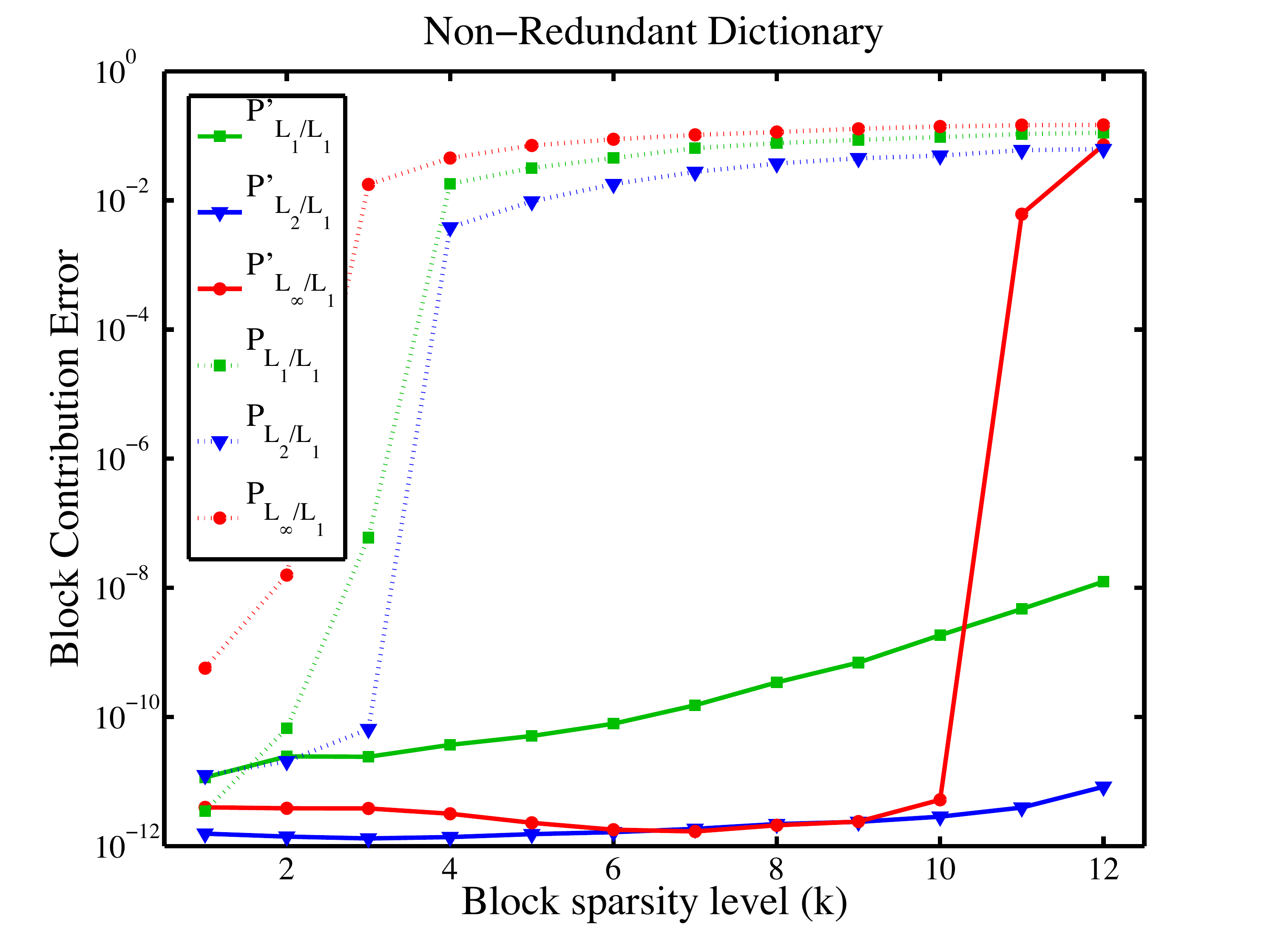} \
\includegraphics[width=0.31\linewidth, trim = 34 19 77 10 , clip]{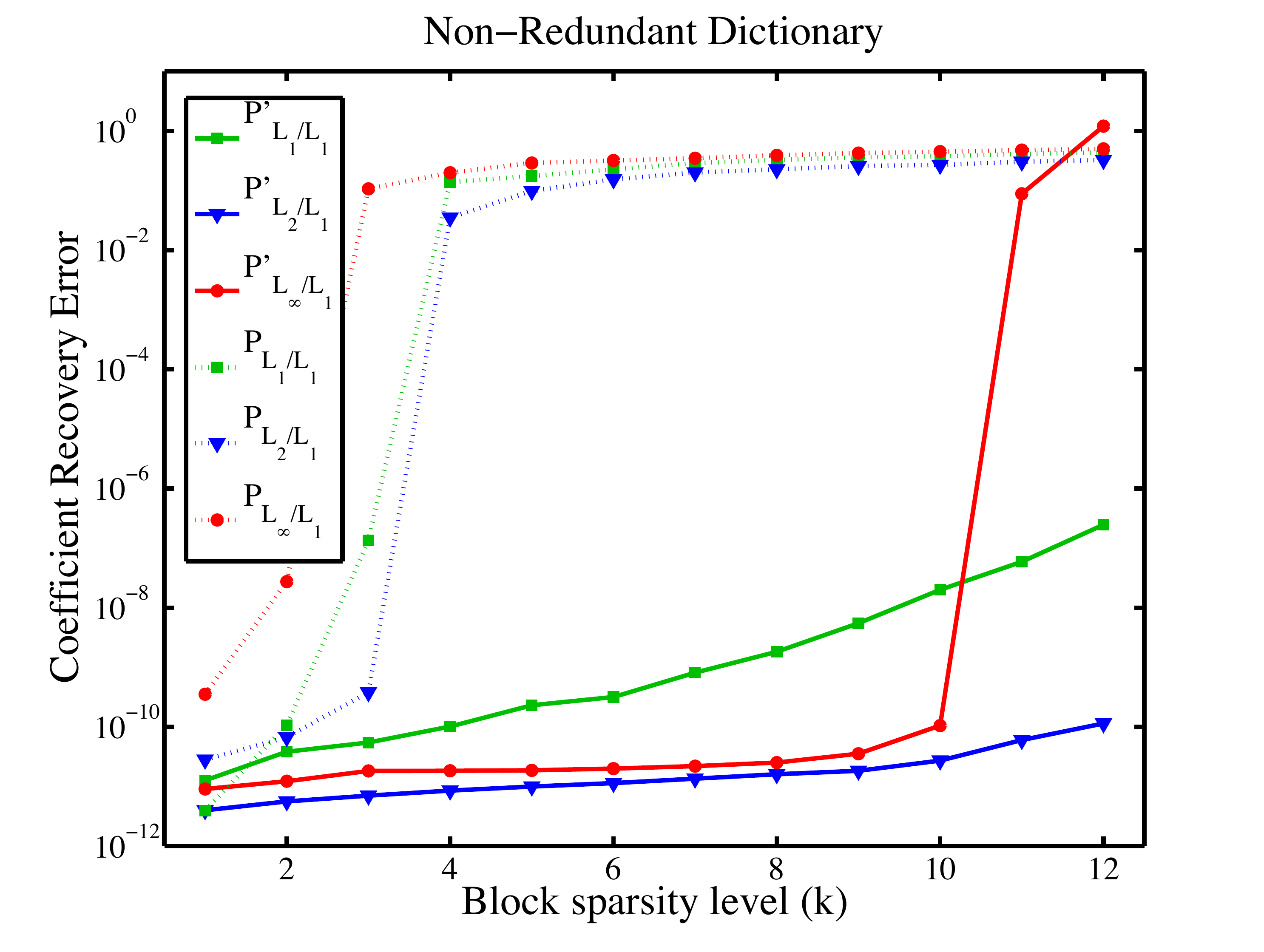}
\vspace{-2mm}
\caption{\footnotesize{Errors of the convex programs on synthetic data with $n=40$, $D = 100$. Reconstruction error (left), block-contribution error (middle) and coefficient recovery error (right) for non-redundant blocks with $m = d = 4$.}}
\label{fig:nonred}
\end{figure*}
\begin{figure*}[t!]
\centering
\includegraphics[width=0.31\linewidth, trim = 34 19 77 10 , clip]{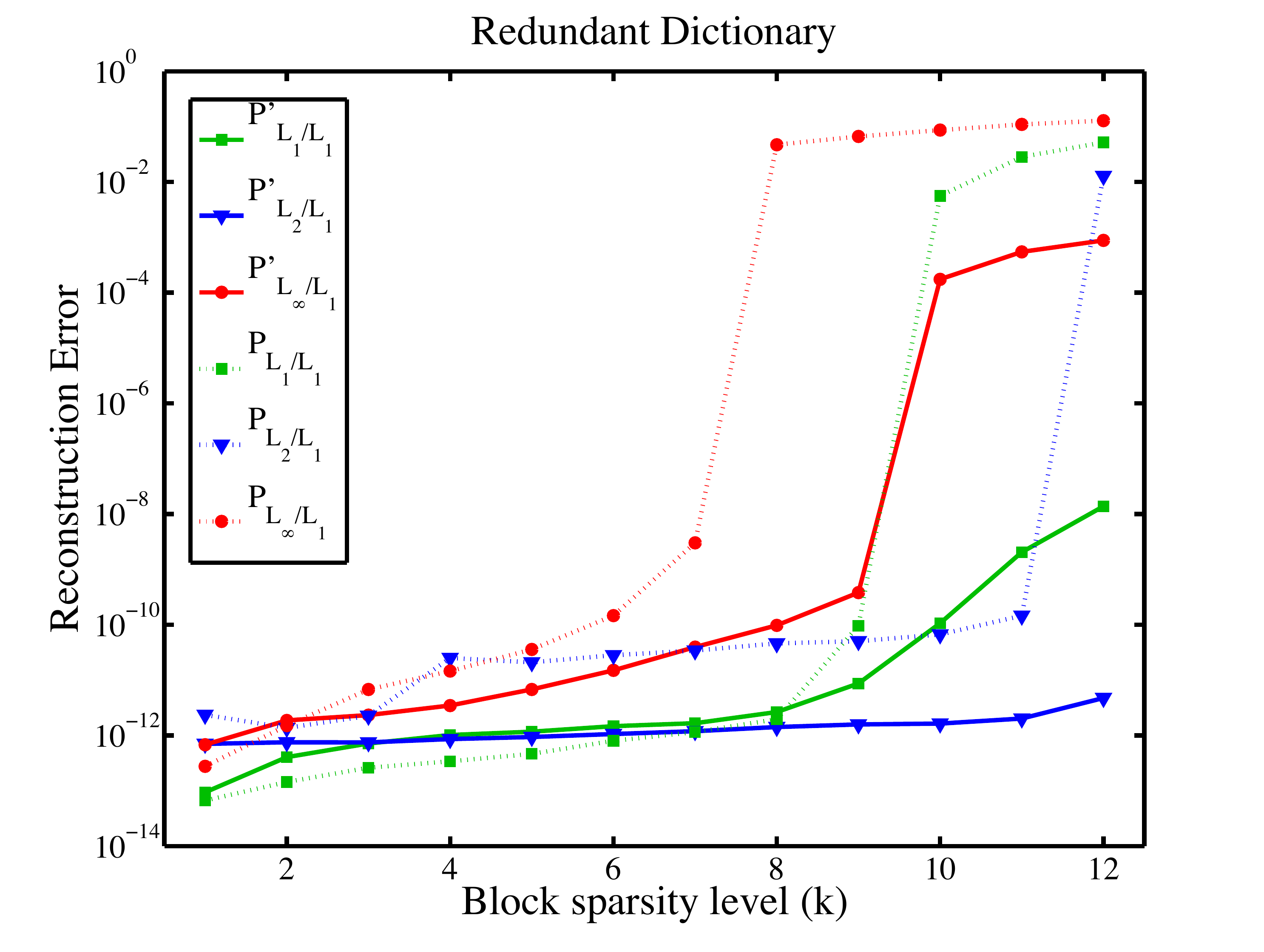} \ \hspace{5mm}
\includegraphics[width=0.31\linewidth, trim = 34 19 77 10 , clip]{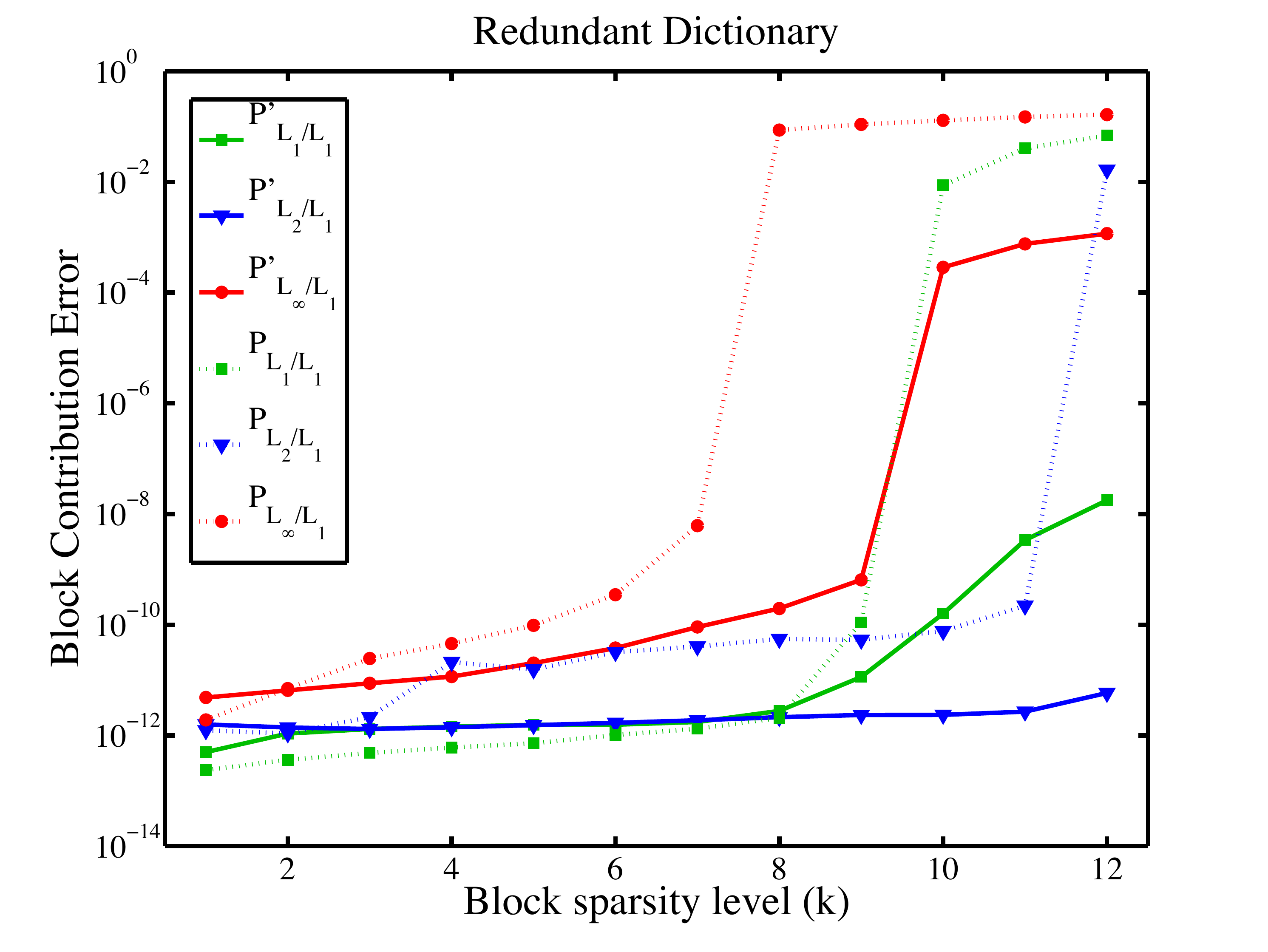}
\vspace{-2mm}
\caption{\footnotesize{Errors of the convex programs on synthetic data with $n=40$, $D = 100$. Reconstruction error (left) and block-contribution error (right) for redundant blocks with $m = 2d = 8$.}}
\label{fig:red}
\end{figure*}
%

\section{Experimental Results}
\label{sec:experiment}
In this section, we evaluate the performance of the two classes of convex programs for recovering block-sparse representations of signals.  
We evaluate the performance of the convex programs through synthetic experiments as well as real experiments on the face recognition problem.


\subsection{Synthetic Experiments}
We consider the problem of finding block-sparse representations of signals in dictionaries whose atoms are drawn from a union of disjoint subspaces. We investigate the performance of the two classes of convex programs for various block-sparsity levels. 

For simplicity, we assume that all the subspaces have the same dimension $d$ and that the blocks have the same length $m$. First, we generate random bases $\A_i \in \Re^{D \times d}$ for $n$ disjoint subspaces $\{\mathcal{S}_i\}_{i=1}^{n}$ in $\Re^D$ by orthonormalizing i.i.d. Gaussian matrices where the elements of each matrix are drawn independently from the standard Gaussian distribution.\footnote{In order to ensure that the generated bases correspond to disjoint subspaces, we check that each pair of bases must be full column-rank.} Next, using the subspace bases, we draw $m \in \{d, 2d\}$ random vectors in each subspace to form blocks $\B[i] \in \Re^{D \times m}$. 
For a fixed block-sprsity level $k$, we generate a signal $\y \in \Re^D$ using a random $k$-block-sparse vector $\c^0 \in \Re^{nm}$ where the $k$ nonzero blocks, $\Lambda_k$, are chosen uniformly at random from the $n$ blocks and the coefficients in the nonzero blocks are i.i.d. and drawn from the standard Gaussian distribution.

For each class of the convex programs $P_{\ell_q/\ell_1}$ and $P'_{\ell_q/\ell_1}$ with $q \in \{1,2,\infty\}$, we measure the following errors. The \emph{reconstruction error} measures how well a signal $\y$ can be reconstructed from the blocks of the optimal solution $\c^*$ corresponding to the correct support $\Lambda_k$ and is defined as
\begin{equation}
\label{eq:RecErr}
\text{reconstruction error} = \frac{ \| \y - \sum_{i \in \Lambda_k}{\B[i]} \c^*[i] \|_2} {\| \y \|_2}.
\end{equation}
Ideally, if an optimization algorithm is successful in recovering the correct vector in each subspace, \ie, $\B[i]\c^*[i] = \B[i]\c^0[i]$ for all $i$, then the reconstruction error is zero.  As we expect that the contribution of the blocks corresponding to $\Lambda_{\widehat{k}}$ to the reconstruction of the given signal be zero, \ie, $\B[i] \c^*[i] = 0$, we measure the \emph{block-contribution error} as
\begin{equation}
\label{eq:BlkErr}
\text{block contribution error} =  1 - \frac{\sum_{i \in \Lambda_k} \| \B[i] \c^*[i] \|_2}{\sum_{i=1}^{n} \| \B[i] \c^*[i] \|_2} \in [0,1].
\end{equation}
The error is equal to zero when all contributing blocks correspond to $\Lambda_k$ and it is equal to one when all contributing blocks correspond to $\Lambda_{\widehat{k}}$. For non-redundant blocks, since $\c^0$ is the unique $k$-block-sparse vector such that $\y = \B \c^0$, we can also measure the \emph{coefficient recovery error} as 
\begin{equation}
\label{eq:CofErr}
\text{coefficient recovery error} = \frac{\| \c^* - \c^0 \|_2} {\| \c^0 \|_2}.
\end{equation}
%

We generate $L_1 = 200$ different sets of $n = 40$ blocks in $\Re^{100}$ and for each set of $n$ blocks we generate $L_2 = 100$ different block-sparse signals. For a fixed block-sparsity level, we compute the average of the above errors for each optimization program over $L = L_1 \times L_2 = 20,000$ trials.\footnote{In order to solve the convex programs, we use the CVX package which can be downloaded from http://cvxr.com/cvx.}

Figure \ref{fig:nonred} shows the average errors for various block-sparsity levels for non-redundant blocks where $m = d = 4$. As the results show, for a fixed value of $q$, $P'_{\ell_q/\ell_1}$ obtains lower reconstruction, block-contribution, and coefficient recovery errors than $P_{\ell_q/\ell_1}$ for all block-sparsity levels. Moreover, while the performance of $P_{\ell_q/\ell_1}$ significantly degrades for block-sparsity levels greater than $3$, $P'_{\ell_q/\ell_1}$ maintains a high performance for a wider range of block-sparsity levels. 

Figure \ref{fig:red} shows the average errors for various block-sparsity levels for redundant blocks with $m = 2d = 8$. Similar to the previous case, for a fixed $q$, $P'_{\ell_q/\ell_1}$ has a higher performance than $P_{\ell_q/\ell_1}$ for all block-sparsity levels. Note that redundancy in the blocks improves the performance of $P_{\ell_q/\ell_1}$. Specifically, compared to the case of non-redundant blocks, the performance of $P_{\ell_q/\ell_1}$ degrades at higher sparsity levels. 

An interesting observation from the results of Figures \ref{fig:nonred} and \ref{fig:red} is that for each class of convex programs, the case of $q=\infty$ either has a lower performance or degrades at lower block-sparsity levels than $q = 1, 2$. In addition, the case of $q=2$ in general performs better than $q=1$.

\begin{figure*}[t!]
\centering
\hspace{0mm}
\includegraphics[width=0.0440\linewidth, trim = 68 30 68 20 , clip]{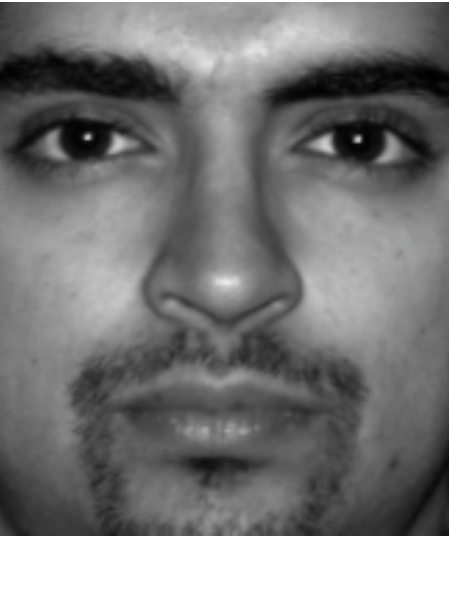}\
\includegraphics[width=0.0440\linewidth, trim = 68 30 68 20 , clip]{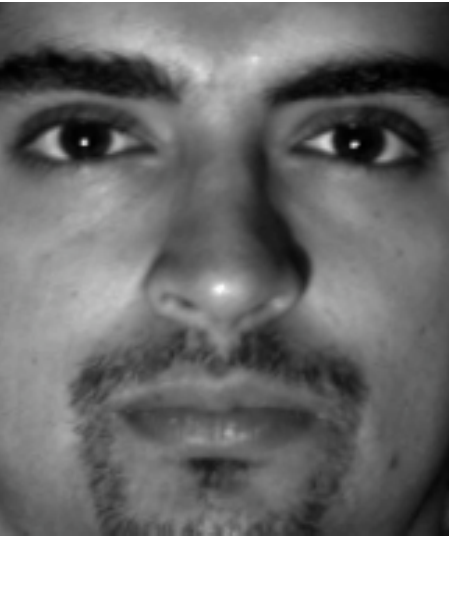}\
\includegraphics[width=0.0440\linewidth, trim = 68 30 68 20 , clip]{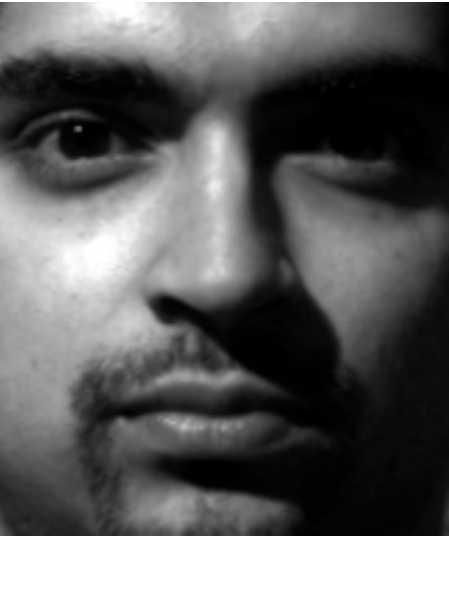}\
\includegraphics[width=0.0440\linewidth, trim = 68 30 68 20 , clip]{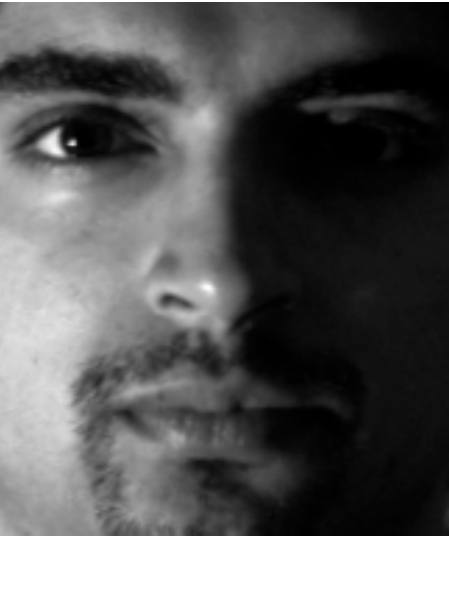}\
\includegraphics[width=0.0440\linewidth, trim = 68 30 68 20 , clip]{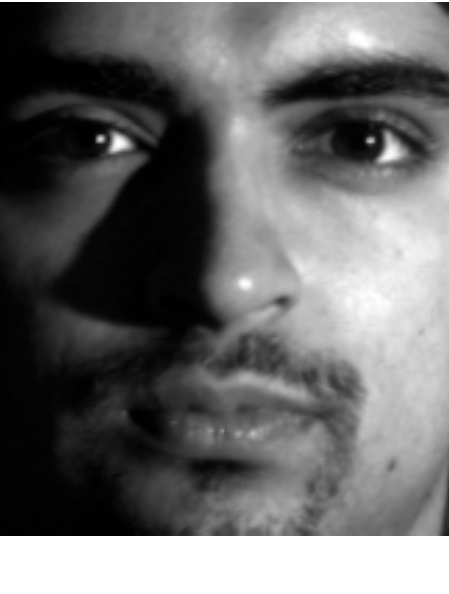}\
\includegraphics[width=0.0440\linewidth, trim = 68 30 68 20 , clip]{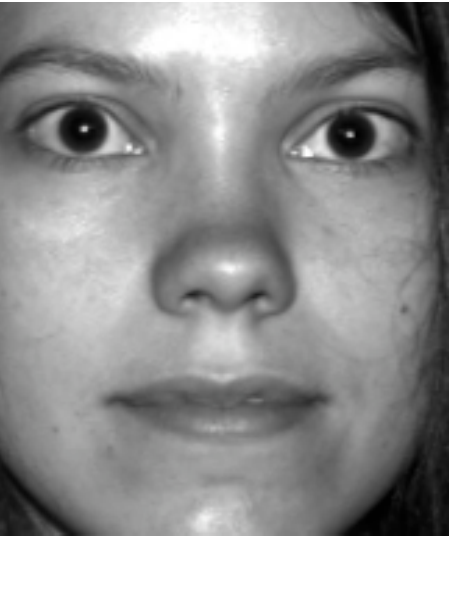}\
\includegraphics[width=0.0440\linewidth, trim = 68 30 68 20 , clip]{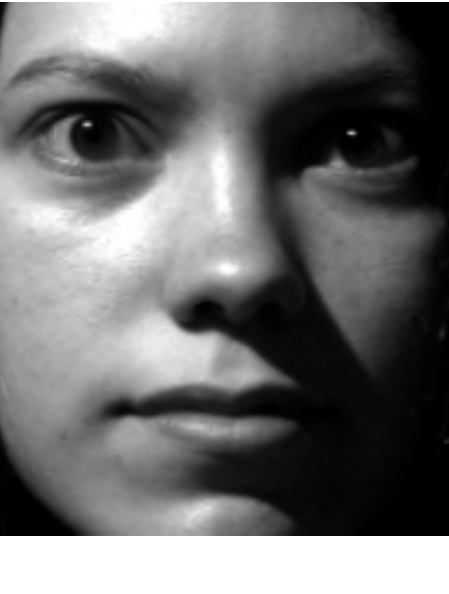}\
\includegraphics[width=0.0440\linewidth, trim = 68 30 68 20 , clip]{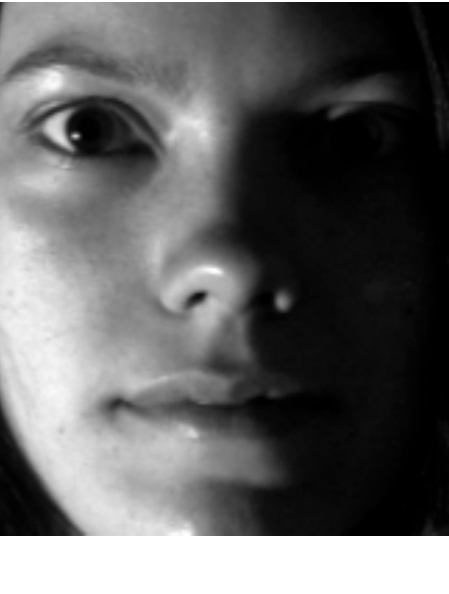}\
\includegraphics[width=0.0440\linewidth, trim = 68 30 68 20 , clip]{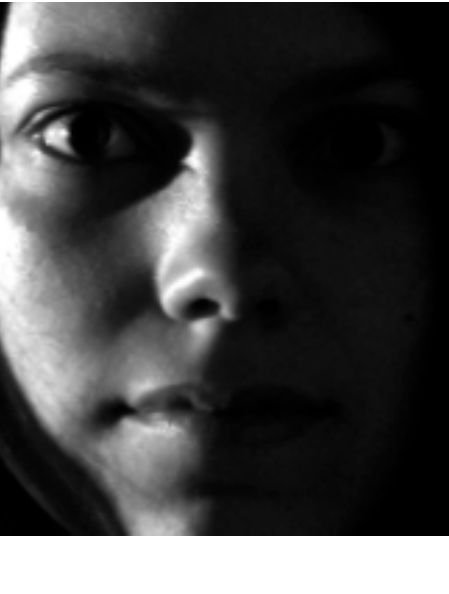}\
\includegraphics[width=0.0440\linewidth, trim = 68 30 68 20 , clip]{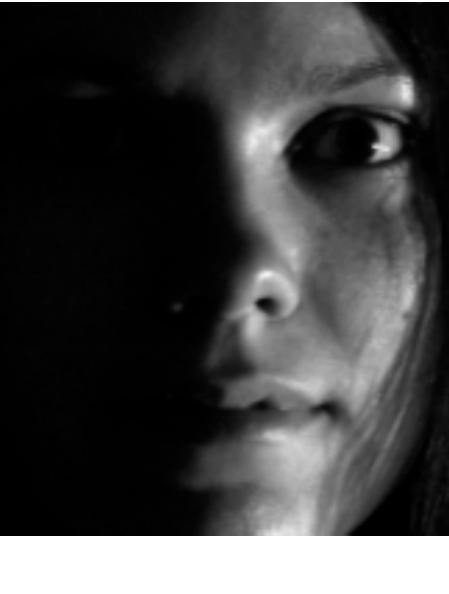}\
\includegraphics[width=0.0440\linewidth, trim = 68 30 68 20 , clip]{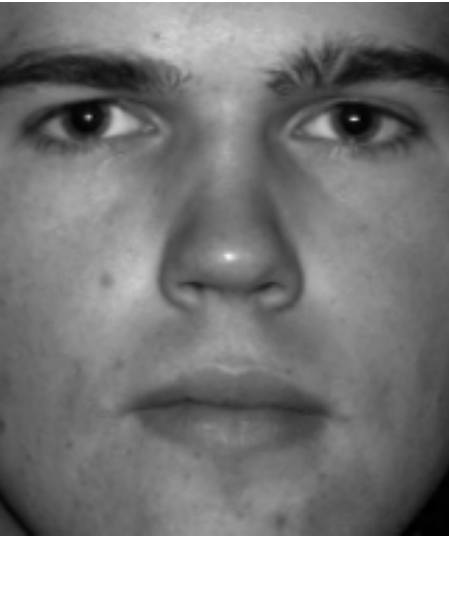}\
\includegraphics[width=0.0440\linewidth, trim = 68 30 68 20 , clip]{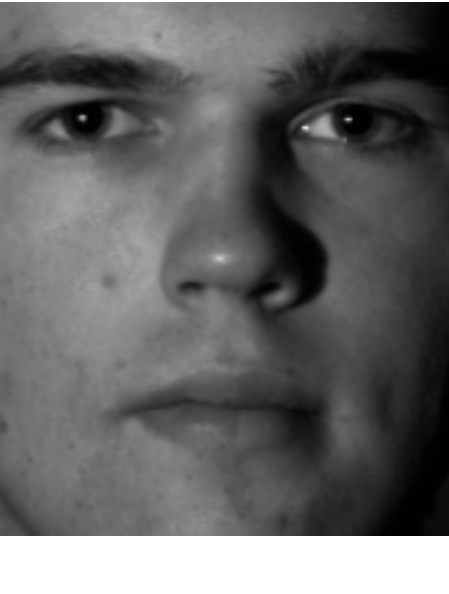}\
\includegraphics[width=0.0440\linewidth, trim = 68 30 68 20 , clip]{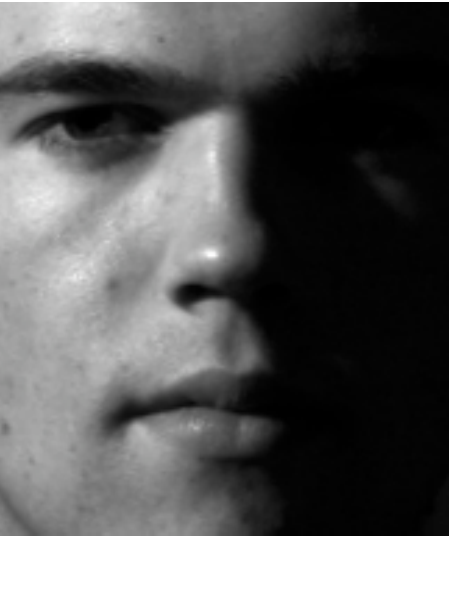}\
\includegraphics[width=0.0440\linewidth, trim = 68 30 68 20 , clip]{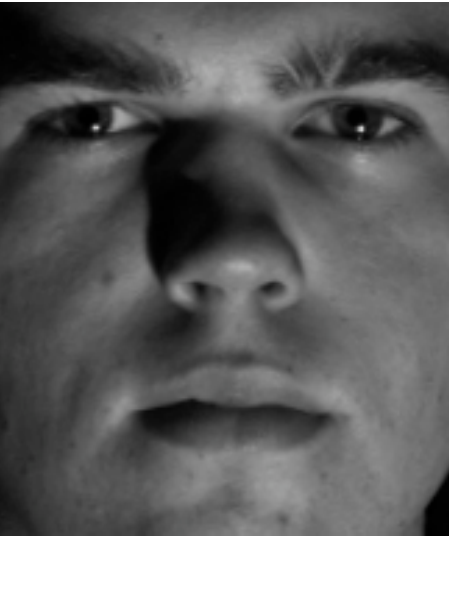}\
\includegraphics[width=0.0440\linewidth, trim = 68 30 68 20 , clip]{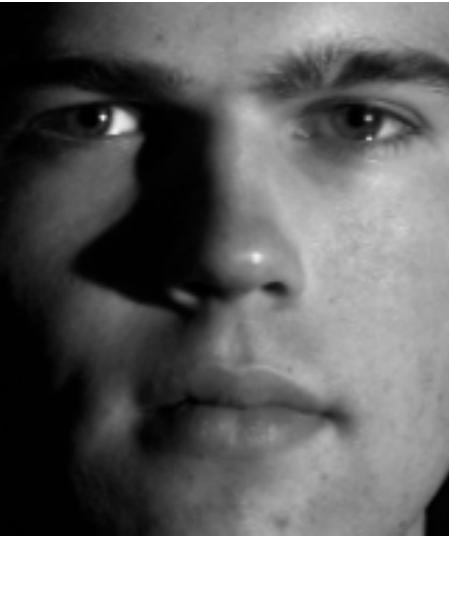}\
\includegraphics[width=0.0440\linewidth, trim = 68 30 68 20 , clip]{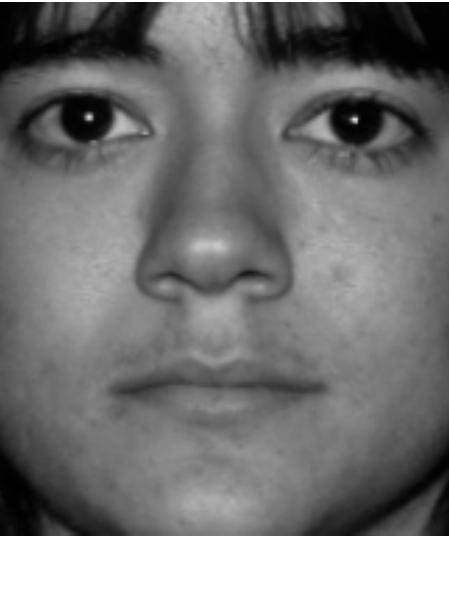}\
\includegraphics[width=0.0440\linewidth, trim = 68 30 68 20 , clip]{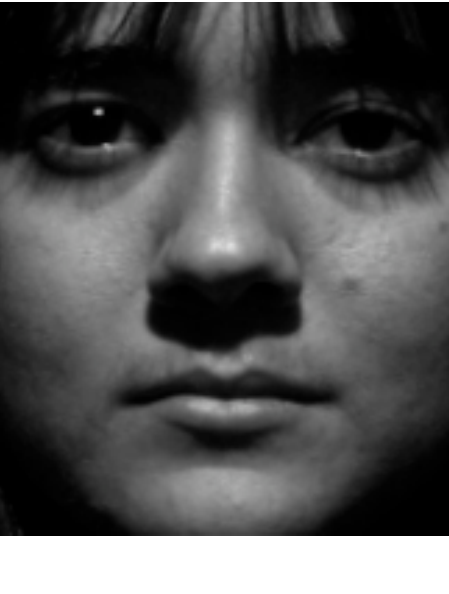}\
\includegraphics[width=0.0440\linewidth, trim = 68 30 68 20 , clip]{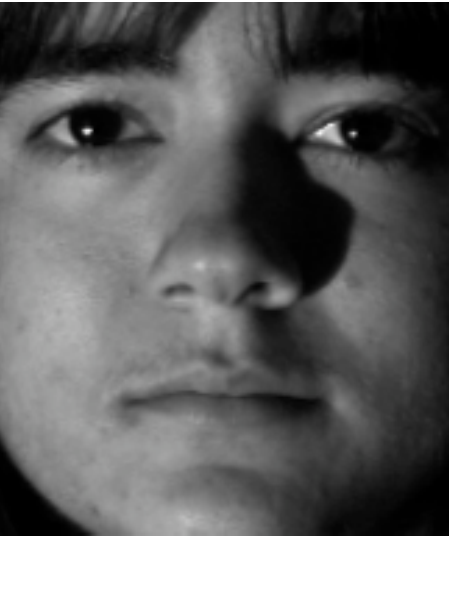}\
\includegraphics[width=0.0440\linewidth, trim = 68 30 68 20 , clip]{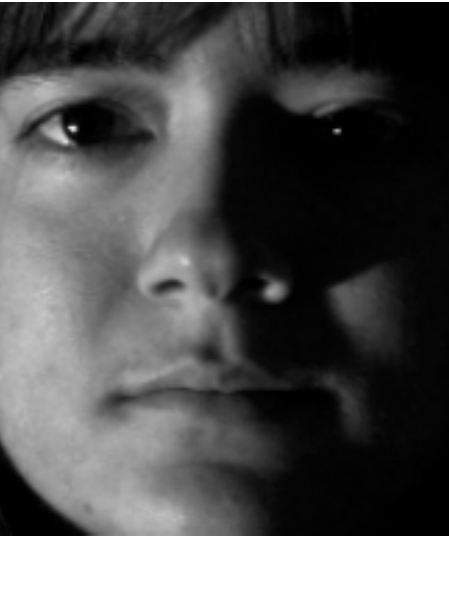}\
\includegraphics[width=0.0440\linewidth, trim = 68 30 68 20 , clip]{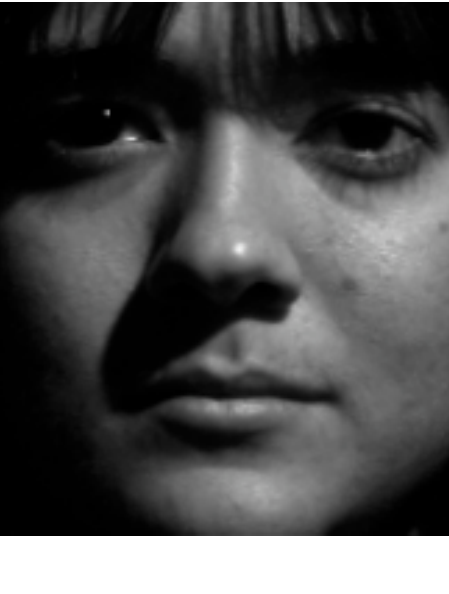}\\ \vspace{1mm}
\includegraphics[width=0.322\linewidth, trim = 24 8 45 3 , clip]{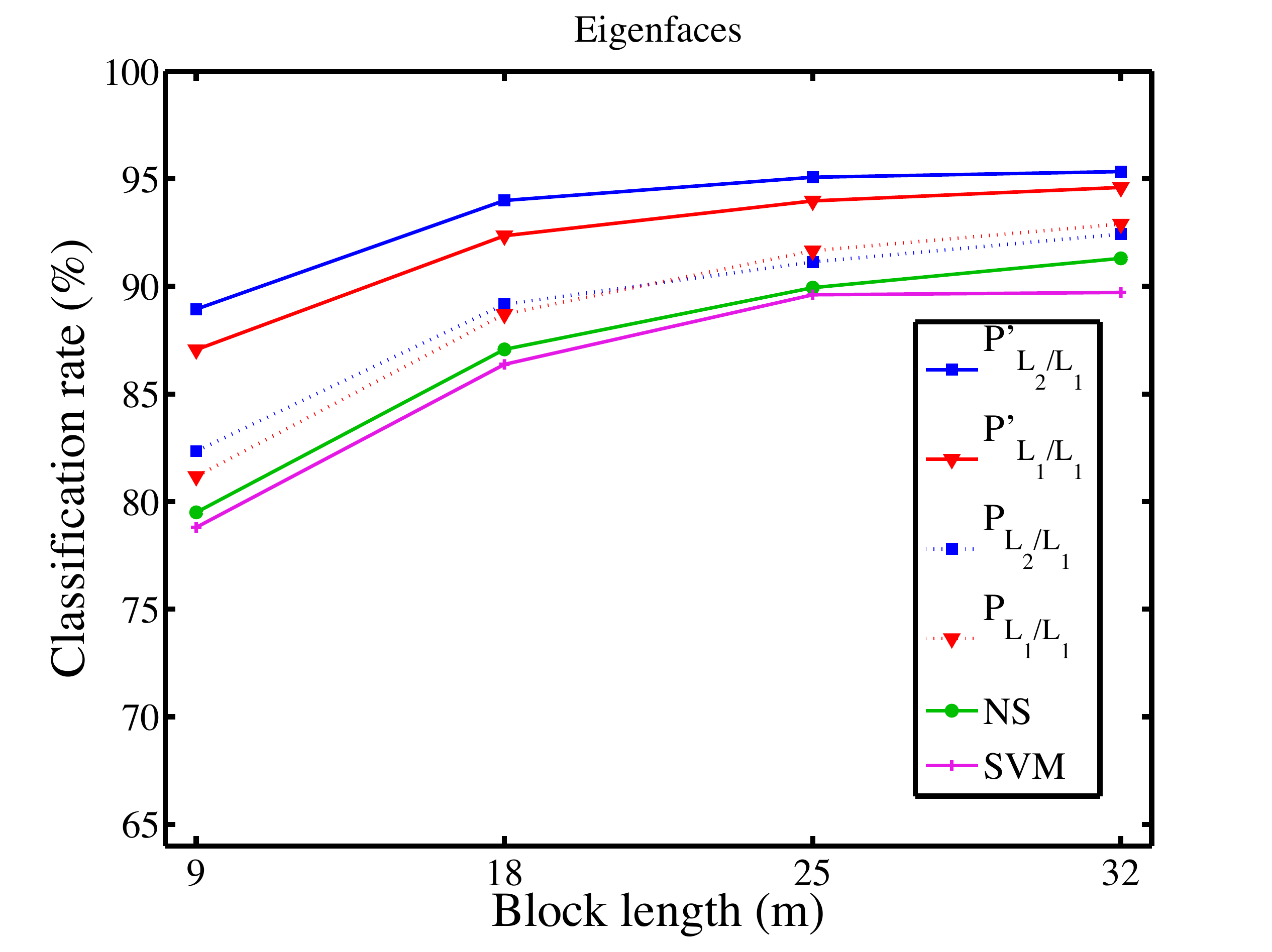} \
\includegraphics[width=0.322\linewidth, trim = 24 8 45 3 , clip]{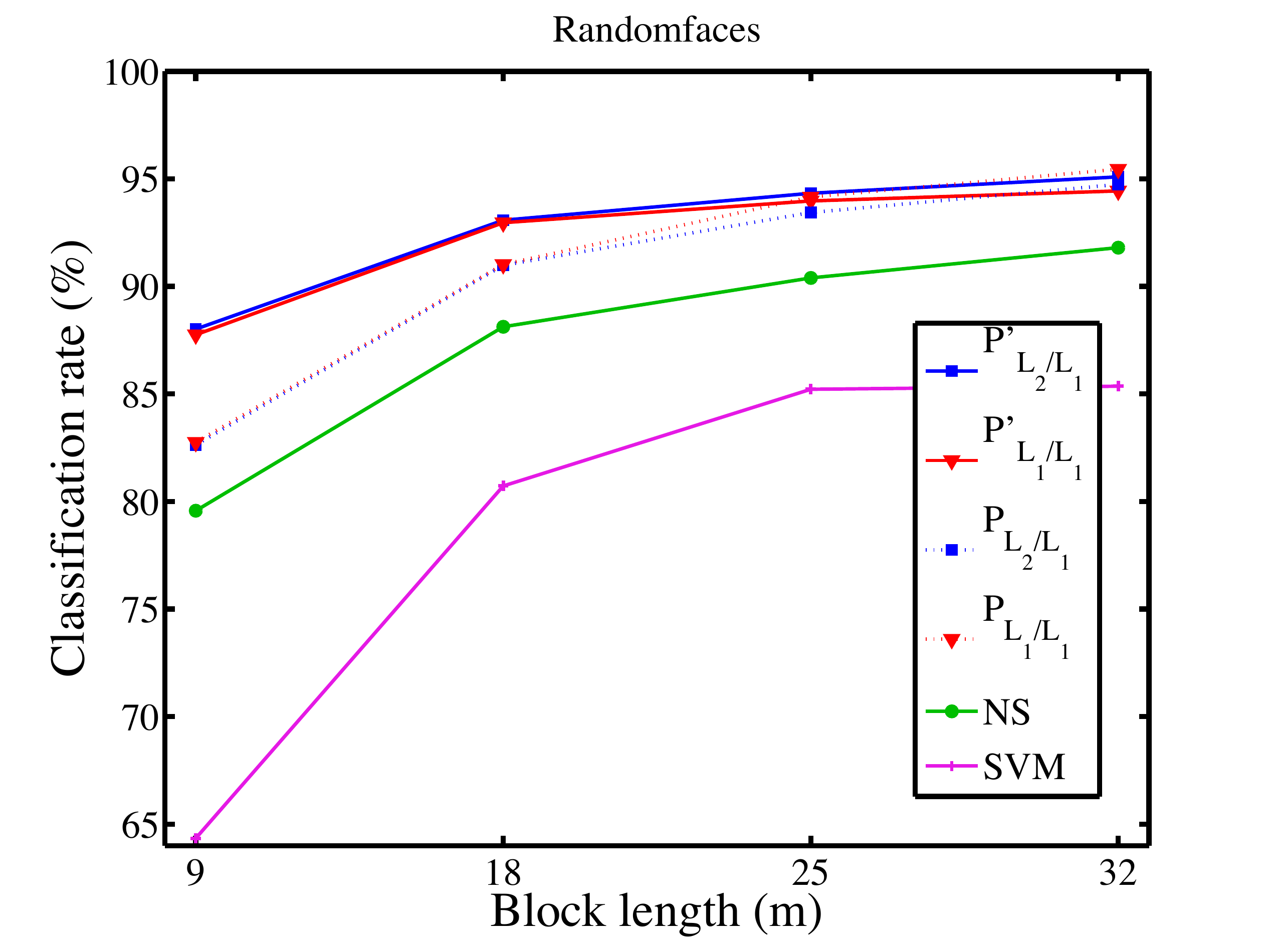} \
\includegraphics[width=0.322\linewidth, trim = 24 8 45 3 , clip]{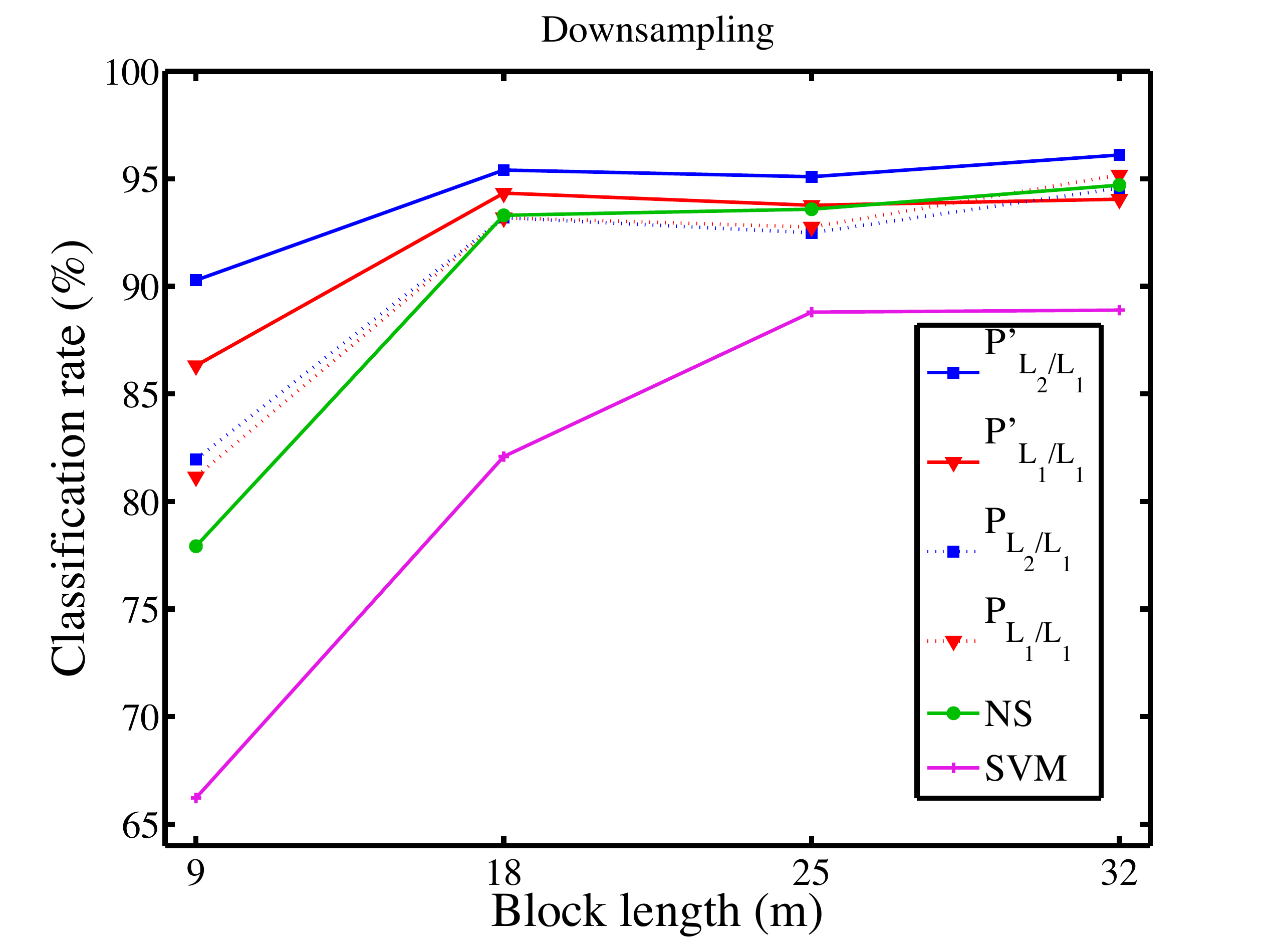} 
\vspace{-2mm} 
\caption{\footnotesize{Top: sample face images from four subjects in the Extended Yale B dataset. 
Bottom: classification rates for the convex programs on the Extended~Yale~B database with $n = 38$ and $D = 132$ as a function of the number of training data in each class. Left: using eigen-faces. Middle: using random projections. Right: using down-sampling.}} \label{fig:facerec-uncorr}
\end{figure*}
%



%

\subsection{Face Recognition}
In this part, we evaluate the performance of the block-sparse recovery algorithms in the real problem of automatic face recognition. Assume we are given a collection of $m n$ face images of $n$ subjects acquired under the same pose and varying illumination. Under the Lambertian assumption, \cite{Basri:PAMI03} shows that the face images of each subject live close to a linear subspace of dimension $d=9$. Thus, the collection of faces of different subjects live close to a union of $9$-dimensional subspaces. 
Let $\b_{ij} \in \Re^D$ denote the $j$-th training image for the $i$-th subject converted into a vector.
We denote the collection of $m$ faces for the $i$-th subject as 
\begin{equation}
\B[i] \triangleq \begin{bmatrix} \b_{i1} & \b_{i2} & \cdots & \b_{im} \end{bmatrix} \in \Re^{D \times m}.
\end{equation}
%
Thus, the dictionary $\B$ consists of the training images of the $n$ subjects. In this dictionary, a new face vector, $\y \in \Re^D$, which belongs to the $i$-th subject, can be written as a linear combination of face vectors from the $i$-th block. However, in reality, a face image is corrupted with cast shadows and specularities. In other words, the columns of $\B$ are corrupted by errors and do not perfectly lie in a low-dimensional subspace. Thus, in the optimization programs, instead of the exact equality constraint $\y = \B \c$, we use the constraint $\| \y - \B \c \|_2 \leq \delta$.\footnote{In all the experiments of this section, we set $\delta = 0.05$.} 
%
%
Following \cite{Wright:PAMI09}, we can find the subject to which $\y$ belongs from
\begin{equation}
\label{eq:faceidentity}
\text{identity}(\y) = \operatorname{arg}\!\min_{i} \, \| \y -  \B[i] \c^*[i] \|_2.
\end{equation}

We evaluate the performance of each one of the above optimization programs on the Extended~Yale~B~dataset \cite{Kriegman:PAMI05}, a few images of which are shown in Figure \ref{fig:facerec-uncorr}. The dataset consists of $2,414$ cropped frontal face images of $n = 38$ individuals. For each subject, there are approximately $64$ face images of size $192 \times 168 = 32,256$, which are captured under various laboratory-controlled lighting conditions. Since the dimension of the original face vectors is very large, we reduce the dimension of the data using the following methods:




\noindent\textbf{--} We use the eigenfaces approach \cite{Turk:CVPR91} by
projecting the face vectors to the first $D$ principal components of
the training data covariance matrix.

\noindent\textbf{--} We multiply the face vectors by a random projection
matrix $\Phi \in \Re^{D \times 32,256}$, which has i.i.d. entries
drawn from a zero mean Gaussian distribution with variance $1\over
D$  \cite{Baraniuk:CA08, Wright:PAMI09}.

\noindent\textbf{--} We down-sample the images by a factor $r$ such that the
dimension of the down-sampled face vectors is $D$.
%

%
%


In the experiments, we set $D = 132$. For each subject, we randomly select $m \in \{9, 18, 25, 32 \}$ training images, to form the blocks $\B[i] \in \Re^{D \times m}$ and use the remaining images for testing. For every test image, we solve each class of the convex programs for $q \in \{ 1, 2 \}$ and determine the identity of the test image using \eqref{eq:faceidentity}.\footnote{Similar to the synthetic experiments, the case of $q = \infty$ has lower performance than other values of $q$, hence we only report the results for $q = 1, 2.$} 
We compute the classification rate as the average number of correctly classified test images for which the recovered identity matches the ground-truth. We repeat this experiment $L=20$ times for random choices of $m$ training data for each subject and compute the mean classification rate among all the trials. 
We compare our results with the nearest subspace (NS) method \cite{Ho:CVPR03} as well as the Linear SVM classifier \cite{Duda:04}.


The recognition results for three dimensionality reduction methods are shown in Figure \ref{fig:facerec-uncorr}. As the results show, the NS and SVM methods have lower performance than methods based on sparse representation. This comes from the fact that the linear SVM assumes that the data in different classes are linearly separable while the face images have a multi-subspace structure, hence are not necessarily separable by a hyperplane. In the case of the NS method, subspaces associated to different classes are close to each other, \ie, have a small principal angle \cite{Elhamifar:TPAMI12}. Since the test images are corrupted by errors, they can be close to the intersection of several subspaces, resulting in incorrect recognition. In addition, using the underlying subspaces ignores the distribution of the data inside the subspaces as opposed to the sparsity-based methods that directly use the training data. On the other hand, for a fixed value of $q$, the convex program $P'_{\ell_q/\ell_1}$ almost always outperforms $P_{\ell_q/\ell_1}$. While the performances of different methods are close for a large number of training data in each class, the difference in their performances becomes evident when the number of data in each class decreases. More specifically, while the performance of all the algorithms degrade by decreasing the number of data in each class, the convex programs $P'_{\ell_q/\ell_1}$ are more robust to decreasing the number of training data. In other words, when the number of training data in each class is as small as the dimension of the face subspace, \ie, $m = d = 9$, $P'_{\ell_q/\ell_1}$ has $5\%$ to $10\%$ higher recognition rate than $P_{\ell_q/\ell_1}$. This result is similar to the result of synthetic experiments, where we showed that the gap between the performance of the two classes of convex programs is wider for non-redundant blocks than redundant blocks. It is also important to note that the results are independent of the choice of the features, \ie, they follow the same pattern for the three types of features as shown in Figure \ref{fig:facerec-uncorr}. In all of them $P'_{\ell_2/\ell_1}$ and $P'_{\ell_1/\ell_1}$ achieve the best recognition results (see \cite{Elhamifar:CVPR11} for experimental results on data with corruption, occlusion, and disguise).

\section{Conclusions}
\label{sec:conc}
We considered the problem of block-sparse recovery using two classes of convex programs, $P_{\ell_q/\ell_1}$ and $P'_{\ell_q/\ell_1}$ and, under a unified framework, we analyzed the recovery performance of each class of convex programs for both non-redundant and redundant blocks. 
%
Interesting avenues of further research include analysis of the stability of each convex program family in the presence of noise in the measured signal as well as generalizing our results to find recovery guarantees for mixed $\ell_q/\ell_p$-norm algorithms \cite{Kowalski:ACHA09}. 
Investigating necessary and sufficient conditions based on the statistical analysis of the projected polytopes \cite{DonohoTanner:AMS09} via the mixed $\ell_q/\ell_1$-norms would also be of great importance. In particular, a geometrical study of the convex programs $P_{\ell_q/\ell_1}$ and $P'_{\ell_q/\ell_1}$ will help in a better understanding of the differences in their block-sparse recovery performance. For dictionaries whose blocks are drawn from certain distributions, the probabilistic analysis of meeting the conditions of Propositions \ref{prop:suff1} and \ref{prop:suff2} as well as Corollaries \ref{cor:suff1} and \ref{cor:suff2} will be the subject of future work. Finally, while there has been a lot of work for fast and efficiently solving the $P_{\ell_q/\ell_1}$ convex program family \cite{Wright:TSP09, Ewout:SJO11}, extending such results to the $P'_{\ell_q/\ell_1}$ family and its unconstrained Lasso-type variations is an interesting open avenue for further research. 


%
\section*{Acknowledgment}
\vspace{0mm}
{This work was partially supported by grants NSF CNS-0931805, NSF ECCS-0941463, NSF OIA-0941362, and ONR N00014-09-10839.}

\section*{Appendix}
\label{sec:appendix}

\subsection{Proof of Proposition \ref{lem:cum-mut}}
\label{app:lemcummut}
\noindent Let $\Lambda_k \triangleq \{j_1^*, \ldots, j_k^* \}$ and $i^* \notin \Lambda_k$ be the set of indices for which $\zeta_k$ is obtained, \ie,
%
\begin{equation}
\zeta_k = \max_{\Lambda_k} \; \max_{i \notin \Lambda_k} \; \sum_{j \in \Lambda_k}{\mu(\S_i,\S_j)} = \sum_{l = 1}^{k}{\mu(\S_{i^*},\S_{j_l^*})}.
\end{equation}
Denoting the sorted subspace coherences among all pairs of different subspaces by $\mu_S = \mu_{1} \geq \mu_{2} \geq \cdots$, 
%
we have
%
\begin{equation}
\zeta_k = \sum_{l = 1}^{k}{\mu(\S_{i^*},\S_{j_l^*})} \leq u_k = \sum_{l=1}^{k}{\mu_l} \leq k \mu_S,
\end{equation}
which proves the desired result.

\vspace{-1mm}
\subsection{Proof of Proposition \ref{prop:uniqueness1} }
\label{app:uniquenessprop}
\noindent We prove this result using contradiction. 

\noindent $(\Longrightarrow)$ Assume there exists a $2k$-block-sparse vector $\bar{\c} \neq 0$ such that $\bar{\B} \, \bar{\c} = 0$. We can write $\bar{\c}^{\top} = \begin{bmatrix} \bar{\c}_1^{\top} & \bar{\c}_2^{\top} \end{bmatrix}$ where $\bar{\c}_1$ and $\bar{\c}_2$ are $k$-block-sparse vectors. So, we have
\begin{equation}
\bar{\B} \, \bar{\c} =  \begin{bmatrix} \bar{\B}_1 & \bar{\B}_2 \end{bmatrix} \!\! \begin{bmatrix} \bar{\c}_1 \\ \bar{\c}_2 \end{bmatrix} = 0 \! \implies \!  \bar{\y} \triangleq \bar{\B}_1 \bar{\c}_1 \!=\! - \bar{\B}_2 \bar{\c}_2.
\end{equation}
Thus, there exists a vector $\bar{\y}$ that has two $k$-block-sparse representations in $\B$ using different sets of blocks. This contradicts the uniqueness assumption of the proposition.
%
%

\noindent
$(\Longleftarrow)$ Assume there exists a vector $\y$ that has two different $k$-block-sparse representations using $(\{i_l\},\{\s_{i_l}\}) \neq (\{i'_l\},\{\s'_{i_l}\})$. 
Since for each block of $\bar{\B}$, we have $\rank(\bar{\B}[i]) = \rank(\B[i])$, there exist $\bar{\c}_1$ and $\bar{\c}_2$ such that $\y = \bar{\B} \, \bar{\c}_1 = \bar{\B} \, \bar{\c}_2,$ where $\bar{\c}_1$ and $\bar{\c}_2$ are different $k$-block-sparse with the indices of their nonzero blocks being $\{i_l\}$ and $\{i'_l\}$, respectively. Also, $\bar{\B}[i_l] \bar{\c}_1[i_l] = \s_{i_l}$ and $\bar{\B}[i_l] \bar{\c}_2[i_l] = \s'_{i_l}$. Thus, we have $\bar{\B} \, (\bar{\c}_1 - \bar{\c}_2) = 0$ that contradicts the assumption of the proposition since $\bar{\c}_1 - \bar{\c}_2$ is a $2k$-block-sparse vector.
%
%

\vspace{0mm}
\subsection{Proof of Corollary \ref{cor:uniqueness1} }
\label{app:uniquenesscor}
%
\noindent We prove the result using contradiction. 

\noindent $(\Longrightarrow)$ Assume there exists $\B_n \in \mathbb{B}_{\tau}(\Lambda_n)$ such that $\rank(\B_n) < 2k$. So, there exists a $2k$-sparse vector $\c_n^{\top} \triangleq \begin{bmatrix} c_n^1 \!& \cdots \!& c_n^n \end{bmatrix}$ such that $\B_n \c_n = \sum_{i=1}^{n}{c_n^i \s_i} = 0$, where $\s_i \in \mathbb{W}_{\tau, i}$ is the $i$-th column of $\B_n$. For each  full column-rank submatrix of $\B[i]$, denoted by $\bar{\B}[i] \in \Re^{D \times d_i}$, there exists a unique $\bar{\c}[i]$ such that $\bar{\B}[i] \bar{\c}[i] = c_n^i \s_i$. Thus, $\bar{\B} \, \bar{\c} = 0$, where $\bar{\B}$ is defined in \eqref{eq:Bbar} and $\bar{\c}^{\top} \triangleq \begin{bmatrix} \bar{\c}[1]^{\top} \!\! & \!\! \cdots & \bar{\c}[n]^{\top} \end{bmatrix}$ is a $2k$-block-sparse vector. This, contradicts the uniqueness assumption using Proposition \ref{prop:uniqueness1}. 

\noindent $(\Longleftarrow)$ Now, assume there exists a signal $\y$ that has two different $k$-block-sparse representations in $\B$. 
From Proposition \ref{prop:uniqueness1}, there exists a $2k$-block-sparse vector $\bar{\c} \neq 0$ such that $\bar{\B} \, \bar{\c} = 0$. We can rewrite $\bar{\B}[i] \, \bar{\c}[i] = c_n^i \s_i$, where $\s_i \in \mathbb{W}_{\tau, i}$. Thus, we have $\bar{\B} \, \bar{\c} = \B_n \c_n = 0$, where $\B_n \triangleq \begin{bmatrix} \s_1 \!& \cdots \!& \s_n \end{bmatrix} \in \mathbb{B}_{\tau}(\Lambda_n)$ and $\c_n^{\top} \triangleq \begin{bmatrix} c_n^1 \!& \cdots \!& c_n^n \end{bmatrix}$ is a $2k$-sparse vector. This implies $\rank(\B_n) < 2k$ that contradicts the assumption.
\vspace{0mm}
\subsection{Proof of Lemma \ref{lem:matrix} }
\label{app:lemmatrix}
\noindent The idea of the proof follows the approach of Theorem 3.5 in \cite{Tropp:TIT04}. Let $\E_k  = \begin{bmatrix} \e_{i_1} \!\!& \cdots \!\! & \e_{i_k} \end{bmatrix} \in \mathbb{B}_{\alpha}(\Lambda_k)$ and $\E_{\widehat{k}} = \begin{bmatrix} \e_{i_{k+1}} \!\!& \cdots \!\! & \e_{i_n} \end{bmatrix}$ where $\| \e_{i_l} \|_2 \leq \sqrt{1+\beta}$ for every $i_l \in \Lambda_{\widehat{k}}$. Using matrix norm properties, we have
\begin{equation}
\label{eq:interm111}
 \| (\E_k^{\top} \E_k)^{-1}\E_k^{\top} \E_{\widehat{k}} \|_{1,1} \leq   \| (\E_k^{\top} \E_k)^{-1} \|_{1,1} \| \E_k^{\top} \E_{\widehat{k}} \|_{1,1}.
\end{equation}
We can write $\E_k^{\top} \E_k = \I_k +  \D$, where
\begin{equation}
\D \triangleq \begin{bmatrix} \e_{i_1}^{\top} \e_{i_1} -1 & \cdots & \e_{i_1}^{\top} \e_{i_k} \\ \vdots & \ddots & \vdots \\ \e_{i_k}^{\top} \e_{i_1} & \cdots & \e_{i_k}^{\top} \e_{i_k}-1  \end{bmatrix}. 
\end{equation}
Since $\E_k \in \mathbb{B}_{\alpha}(\Lambda_k)$, for any column of $\E_k$, we have $\| \e_i \|_2^2 \leq 1+\alpha$. Also, for any two columns $\e_i$ and $\e_j$ of $\E_k$ we have
\begin{equation}
| \e_i^{\top} \e_j | \leq \| \e_i \|_2 \| \e_j \|_2 \, \mu(\mathcal{S}_i,\mathcal{S}_j) \leq (1+\alpha) \, \mu(\mathcal{S}_i,\mathcal{S}_j).
\end{equation}
%
Thus, we can write 
\begin{equation}
\| \D \|_{1,1} \leq \alpha + (1+\alpha) \, \zeta_{k-1}.
\end{equation}
If $\| \D \|_{1,1} < 1$, we can write $(\E_k^{\top} \E_k)^{-1} = (\I_k + \D)^{-1} = \sum_{i=0}^{\infty}{ (-\D)^k / k! }\,$ from which we obtain
\begin{multline}
\label{eq:suffstr3}
\| ( \E_k^{\top} \E_k )^{-1} \|_{1,1} \leq \sum_{i=0}^{\infty}{ \frac{\| \D \|_{1,1}^k}{k!} } = \frac{1}{1-\| \D \|_{1,1}} \\\leq \frac{1}{ 1 - [ \,\alpha + (1+\alpha) \zeta_{k-1} \, ] }.
\end{multline}
On the other hand, $\E_k^{\top} \E_{\widehat{k}}$ has the following form 
\begin{equation}
\E_k^{\top} \E_{\widehat{k}} = \begin{bmatrix} \e_{i_1}^{\top} \e_{i_{k+1}} & \cdots & \e_{i_1}^{\top} \e_{i_n} \\ \vdots & \ddots & \vdots \\ \e_{i_k}^{\top} \e_{i_{k+1}} & \cdots & \e_{i_k}^{\top} \e_{i_n}  \end{bmatrix}. 
\end{equation}
Since for each column $\e_i$ of the matrix $\E_k$ we have $\| \e_i \|_2^2 \leq 1+\alpha$ and for each column $\e_j$ of the matrix $\E_{\widehat{k}}$ we have $\| \e_j \|_2^2 \leq 1+\beta$, we obtain 
\begin{equation}
\label{eq:suffstr2}
\| \E_k^{\top} \E_{\widehat{k}}  \|_{1,1} \leq \sqrt{(1+\alpha) (1+\beta)} \; \zeta_k.
\end{equation}
Finally, substituting \eqref{eq:suffstr3} and \eqref{eq:suffstr2} into \eqref{eq:interm111}, we get
\begin{multline}
 \| (\E_k^{\top} \E_k)^{-1} \E_k^{\top} \E_{\widehat{k}} \|_{1,1} \leq   \| (\E_k^{\top} \E_k)^{-1} \|_{1,1} \| \E_k^{\top} \E_{\widehat{k}} \|_{1,1} \\ \leq \frac{ \sqrt{(1+\alpha) (1+\beta)} \, \zeta_{k} }{ 1 - [\, \alpha + (1+\alpha) \zeta_{k-1} \, ] }.
\end{multline}
%

%

\subsection{Proof of Proposition \ref{prop:suff2}}
\label{app:propsuff2}
\noindent Fix a set $\Lambda_k = \{i_1,\ldots,i_k\}$ of $k$ indices from $\{1, \ldots, n\}$ and denote by $\Lambda_{\widehat{k}} = \{i_{k+1},\ldots,i_n\}$ the set of the remaining indices. For a signal $\x$ in the intersection of $\oplus_{i \in \Lambda_k} \mathcal{S}_i$ and $\oplus_{i \in \Lambda_{\widehat{k}}} \mathcal{S}_i$, let $\breve{\c}^*$ be the solution of the optimization program \eqref{eq:L2L1red3}. We can write 
\begin{equation}
\label{eq:op1p}
\x = \sum_{i \in \Lambda_k} \B[i] \c^*[i] = \B_k \a_k,
\end{equation}
where $\B_k \triangleq \begin{bmatrix} \s_{i_1} \!\!& \ldots \!\!& \s_{i_k} \end{bmatrix}$ and $\a_k \triangleq \begin{bmatrix} a_{i_1} \!\!& \ldots \!\!& a_{i_k} \end{bmatrix}^{\top}$ are defined as follows. For every $i \in \Lambda_k$, if $\breve{\c}^*[i] \neq 0$ and $\B[i] \breve{\c}^*[i] \neq 0$, define 
\begin{equation}
\s_i \triangleq \frac{\B[i] \breve{\c}^*[i]}{ \| \B[i]  \breve{\c}^*[i] \|_q }, \quad a_i \triangleq \| \B[i] \breve{\c}^*[i] \|_q. 
\end{equation}
Otherwise, let $\s_i$ be an arbitrary vector in $\mathcal{S}_i$ of unit Euclidean norm and $a_i = 0$. According to Definition \ref{def:normequiv}, we have $\B_k \in \mathbb{B}_{\epsilon'_q}(\Lambda_k)$.

Now, let $\widehat{\c}^*$ be the solution of the optimization program \eqref{eq:L2L1red4}. We can write
\begin{equation}
\label{eq:op2p}
\x = \sum_{i \in \Lambda_{\widehat{k}}} \B[i] \widehat{\c}^*[i] = \B_{\widehat{k}} \a_{\widehat{k}},
\end{equation}
where $\B_{\widehat{k}} \triangleq \begin{bmatrix} \s_{i_{k+1}} \!\!& \ldots \!\!& \s_{i_n} \end{bmatrix}$ and $\a_{\widehat{k}} \triangleq \begin{bmatrix} a_{i_{k+1}} \!\!& \ldots \!\!& a_{i_n} \end{bmatrix}^{\top}$ are defined in the following way. For every $i \in \Lambda_{\widehat{k}}$, if $\widehat{\c}^*[i] \neq 0$ and $\B[i] \widehat{\c}^*[i] \neq 0$, define 
\begin{equation}
\s_i \triangleq \frac{\B[i] \widehat{\c}^*[i]}{ \| \B[i] \widehat{\c}^*[i] \|_q }, \quad a_i \triangleq \| \B[i] \widehat{\c}^*[i] \|_q. 
\end{equation}
Otherwise, let $\s_i$ be an arbitrary vector in $\mathcal{S}_i$ of unit Euclidean norm and $a_i = 0$. 
%
%
Note that from Definition \ref{def:normequiv}, we have $\B_{\widehat{k}} \in \mathbb{B}_{\epsilon'_q}(\Lambda_{\widehat{k}})$.

Since $\B_k \in \mathbb{B}_{\epsilon'_q}(\Lambda_k)$, assuming $\epsilon'_q \in [0,1)$, the matrix $\B_k$ is full column-rank from Corollary \ref{cor:uniqueness1}. Hence, we have $\, \a_k = (\B_k^{\top} \B_k)^{-1}\B_k^{\top} \y \,$ and consequently,
\begin{equation}
\| \a_k \|_1 = \| (\B_k^{\top} \B_k)^{-1} \B_k^{\top} \x \|_1.
\end{equation}
Substituting $\x$ from \eqref{eq:op2p} in the above equation, we obtain
\begin{multline}
\label{eq:interm1p}
\| \a_k \|_1 = \| (\B_k^{\top} \B_k)^{-1} \B_k^{\top} \B_{\widehat{k}} \a_{\widehat{k}} \|_1 \\ \leq \| (\B_k^{\top} \B_k)^{-1}\B_k^{\top} \B_{\widehat{k}} \|_{1,1} \| \a_{\widehat{k}} \|_1.
\end{multline}
Using Lemma \ref{lem:matrix} with $\alpha = \epsilon'_q$ and $\beta = \epsilon'_q$, we have 
\begin{equation}
\| ( \B_k^{\top} \B_k)^{-1} \B_k^{\top} \B_{\widehat{k}} \|_{1,1} \leq \frac{ (1 + \epsilon'_q) \zeta_k }{ 1 - [ \epsilon'_q + (1+\epsilon'_q) \zeta_{k-1}] } .
\end{equation}

Thus, if the right hand side of the above equation is strictly less than one, \ie, the sufficient condition of the proposition is satisfied, then from \eqref{eq:interm1p} we have $\| \a_k \|_1 < \| \a_{\widehat{k}} \|_1$. Finally, using the definitions of $\a_k$ and $\a_{\widehat{k}}$, we obtain
\begin{equation}
\begin{split}
\sum_{i \in \Lambda_k}{\| \B[i] \breve{\c}^*[i] \|_q} = \| \a_k \|_1  < \| \a_{\widehat{k}} \|_1 = \sum_{i \in \Lambda_{\widehat{k}}}{\| \B[i] \widehat{\c}^*[i] \|_q},
\end{split}
\end{equation}
which implies that the condition of Theorem \ref{thm:SuffRedundant2} 
is satisfied. Thus, $P'_{\ell_q/\ell_1}$ is equivalent to $P'_{\ell_q/\ell_0}$. 
\ifCLASSOPTIONcaptionsoff
  \newpage
\fi

{\small
\footnotesize
\bibliographystyle{IEEEtran}
\bibliography{biblio/sparse,biblio/vidal,biblio/learning,biblio/vision,biblio/recognition,biblio/segmentation}
}

\end{document}